\documentclass[11pt,a4paper]{article}
\usepackage[T1]{fontenc}
\usepackage[utf8]{inputenc}
\usepackage{lmodern,amsmath,amssymb,amsthm,mathtools}
\usepackage[margin=20mm]{geometry}
\usepackage{enumitem,booktabs,array,tikz,needspace,float}
\usetikzlibrary{arrows.meta,positioning}
\usepackage[numbers,sort&compress]{natbib}
\usepackage[colorlinks=true,linkcolor=blue,citecolor=blue,urlcolor=blue]{hyperref}
\hypersetup{pdftitle={Linear circumference in vertex-transitive graphs}}
\newtheorem{theorem}{Theorem}[section]
\newtheorem{lemma}[theorem]{Lemma}
\newtheorem{proposition}[theorem]{Proposition}

\theoremstyle{definition}

\theoremstyle{remark}

\newtheoremstyle{claimstyle}{\medskipamount}{\medskipamount}
  {\normalfont}{}{\bfseries}{.}{.5em}{}
\theoremstyle{claimstyle}
\newtheorem{claim}{Claim}

\newcommand{\Aut}{\operatorname{Aut}}
\setlist{itemsep=2pt,topsep=4pt}
\allowdisplaybreaks[2]
\begin{document}
\title{Linear circumference in vertex-transitive graphs}
\author{Jie Ma\thanks{School of Mathematical Sciences, University of Science and Technology of China, Hefei, Anhui 230026, China, and Yau Mathematical Sciences Center, Tsinghua University, Beijing 100084, China.
Research supported by National Key Research and Development Program of China 2023YFA1010201 and National Natural Science Foundation of China grant 12125106. Email: {\tt jiema@ustc.edu.cn}.}
\and Ziyuan Zhao\thanks{School of Mathematical Sciences, University of Science and Technology of China, Hefei, Anhui 230026, China. Research supported by Innovation Program for Quantum Science and Technology 2021ZD0302902. Email: {\tt zyzhao2024@mail.ustc.edu.cn}.}
}

\date{}
\maketitle
\begin{abstract}
We prove that there is an absolute constant $c>0$ such that every connected vertex-transitive graph $G$ on $n \ge 3$ vertices contains a cycle of length at least $cn$.
Moreover, every such graph with sufficiently large degree $d$ contains a cycle of length at least $(1-d^{-1/100})n$.
This gives the first linear bound towards Lov\'asz's Hamiltonicity conjecture.
The proof combines a structural result of DeVos and Mohar on vertex-transitive graphs with a general framework for finding long cycles, which may be of independent interest.
\end{abstract}

\section{Introduction}\label{sec:introduction}
A graph is \emph{vertex-transitive} if its automorphism group acts transitively on its vertices. In 1969, Lov\'asz~\cite{Lovasz} asked whether every connected vertex-transitive graph has a Hamilton path, a problem now known as the Lov\'asz Hamiltonicity conjecture. 
A stronger conjecture, commonly attributed to Thomassen~\cite{GLSTY}, states that every sufficiently large connected vertex-transitive graph has a Hamilton cycle.
The special case of Cayley graphs has an even earlier history, with work of Rankin~\cite{Rankin} and Rapaport--Strasser~\cite{Rapaport} predating these conjectures.

There has been extensive research on improving the general lower bound for the length of longest cycles in connected vertex-transitive graphs.
In an influential paper,
Babai~\cite{Babai} proved that every connected vertex-transitive graph on $n\ge3$ vertices has a cycle of length $\Omega(\sqrt n)$.
This remained the best general bound for more than forty years, until DeVos~\cite{DeVos} improved it in 2023 to $\Omega(n^{3/5})$, revealing an interesting relation between the length and transversal of longest cycles.
Since then, the problem has attracted considerable attention.
Groenland, Longbrake, Steiner, Turcotte and Yepremyan~\cite{GLSTY} first improved the bound to $\Omega(n^{13/21})$, and Norin, Steiner, Thomass\'e and Wollan~\cite{NSTW} subsequently obtained $\Omega(n^{9/14})$.
Buci\'c, Christoph, Pokrovskiy and Steiner~\cite{BCPS} then proved a bound of $n^{2/3-o(1)}$.
In a subsequent working draft (available at~\cite{BCPSAsymptotic}), they improved this to $n^{1-o(1)}$, a bound that was also proved by Li and Methuku~\cite{LiMethuku}.
Several of these works rely on estimates for the intersections and transversals of longest paths and cycles; see~\cite{GLSTY,NSTW,MaZhao,BCPS}.

Hamiltonicity and almost Hamiltonicity are known under additional degree assumptions. 
Christofides, Hladk\'y and M\'ath\'e~\cite{CHM} proved that for every $\varepsilon>0$, every sufficiently large connected vertex-transitive graph on $n$ vertices with degree at least $\varepsilon n$ has a Hamiltonian cycle.
For Cayley graphs, Bedert, Dragani\'c, M\"uyesser and Pavez-Sign\'e~\cite{BDMP} weakened
the degree assumption: for some absolute $c>0$, every sufficiently large connected
Cayley graph of degree at least $n^{1-c}$ has a Hamiltonian cycle.
Very recently, Christoph, Hunter and Sudakov~\cite{CHS} proved that every connected $d$-regular vertex-transitive graph with $d=\omega(\log^3 n\log\log n)$ contains a cycle of length $(1-o(1))n$.

In this paper, we prove the first linear bound towards the Lov\'asz Hamiltonicity conjecture. 

\begin{theorem}\label{thm:main}
There exist an absolute constant $c>0$ and an integer $d_0$
such that every finite connected vertex-transitive graph on
$n\ge3$ vertices contains a cycle of length at least $cn$.
Moreover, if its degree $d$ is at least $d_0$, then it contains
a cycle of length at least $(1-d^{-1/100})n$.
\end{theorem}

In particular, the same result holds for all connected Cayley graphs on at least three vertices.
The proof combines a general criterion for finding long cycles (which we will introduce shortly) with
a structural theorem of DeVos and Mohar~\cite{DeVosMohar}.
The latter leads to a quotient graph of the given vertex-transitive graph that is a Cayley graph of a
cyclic or a dihedral group; an outline of the proof is given in Subsection~\ref{subsec:pf}.

\subsection{A general criterion for linear circumference}
We prove the following general criterion for
the existence of long cycles.
Both its formulation and its proof are motivated by the recent work of Segal and Verstra\"ete~\cite{SegalVerstraete2026}.

\begin{theorem}\label{thm:general}
For every $K\ge0$, there are a constant $c>0$ and integers
$n_0,d_0$ with the following property.
Let $G$ be a graph on $n\ge n_0$ vertices.
Suppose that for every $S\subseteq V(G)$, the components
$X_1,\ldots,X_t$ of $G-S$, indexed in nonincreasing order of size,
satisfy
\begin{equation}\label{eq:component-condition}
 \sum_{j=2}^t |X_j|^{1/8}\le K|S|.
\end{equation}
Then $G$ contains a cycle of length at least $cn$.
Moreover, if $G$ is $d$-regular with $d\ge d_0$, then it contains
a cycle of length at least $(1-d^{-1/100})n$.
\end{theorem}

We outline its proof in Subsection~\ref{subsec:pf}.
More generally, for a positive function $h$ on the positive integers and a constant $K\ge0$, one can consider the condition
\begin{equation}\label{eq:weighted-components}
 \sum_{j=2}^t h(|X_j|)\le K|S|
\end{equation}
for every $S\subseteq V(G)$, where $X_1,\ldots,X_t$ are the components of $G-S$, listed in nonincreasing order of size.
Condition~\eqref{eq:component-condition} is then precisely the case $h(x)=x^{1/8}$.
The choices $h\equiv1$ and $h(x)=x$ correspond, up to constant
factors in the parameters, to toughness and vertex expansion,
respectively; see Appendix~\ref{app:toughness-expansion}
for the precise comparisons.
The hypothesis of Theorem~\ref{thm:general} can thus be viewed
as interpolating between these two conditions

A similar interpolating connection between vertex connectivity and vertex expansion was studied by Brandt, Broersma, Diestel and Kriesell~\cite{BBDK} in their work on cycles of linear length.

\paragraph{Notation.}
Unless stated otherwise, all graphs considered here are finite, simple, and undirected.
The \emph{length} of a path or cycle is its number of edges.
Given a graph $G$, we write $|G|$ for its order, i.e., the number of vertices, and $c(G)$ for its \emph{circumference}, the maximum length of a cycle in $G$ (with $c(G)=0$ if $G$ is acyclic).
For a nonnegative integer $r$ and a vertex $v\in V(G)$, the \emph{ball of radius $r$ centred at $v$} is
$B_G(v,r)=\{u\in V(G):d_G(u,v)\le r\}$, where $d_G(u,v)$ denotes the distance between $u$ and $v$, i.e., the length of a shortest $(u,v)$-path.
If $G$ is vertex-transitive, we write $\beta_r=|B_G(v,r)|$, which is independent of the choice of $v$.
All implicit constants in asymptotic notation are absolute, unless they depend on parameters explicitly fixed in the statement concerned.

Subgraphs are \emph{disjoint} if their vertex sets are disjoint. Given disjoint sets $A,B\subseteq V(G)$, an \emph{$(A,B)$-path} is a path with one endpoint in $A$, the other in $B$, and no internal vertex in $A\cup B$; we use the same terminology for subgraphs, meaning their vertex sets.

\subsection{Proof outline}\label{subsec:pf}
We first outline the proof of Theorem~\ref{thm:general}.
Let $G$ be a graph satisfying condition~\eqref{eq:component-condition}.
We begin by constructing a family of disjoint cycles whose total length is linear in $n$ (i.e., Lemma~\ref{lem:prefix-initial}).
For general graphs satisfying condition~\eqref{eq:component-condition},
we use a DFS argument to give such
a family. For regular graphs, we use Theorem~\ref{thm:few-cycles} instead
to find disjoint cycles covering almost all vertices.
The most technical part of the proof is Lemma~\ref{lem:prefix-bounded},
whose proof adapts an argument of Segal and
Verstra\"ete~\cite{SegalVerstraete2026}.
Given a family of at least two sufficiently long disjoint cycles,
Lemma~\ref{lem:prefix-bounded} gives another family whose minimum cycle length is
at least $3/2$ times that of the given family, with only
a small decrease in total length.
To prove Lemma~\ref{lem:prefix-bounded}, we first find
disjoint paths joining distinct cycles of the given family, with many endpoints on every cycle.
Their existence follows from Lemma~\ref{lem:ports},
which uses condition~\eqref{eq:component-condition} and an application
of Mader's theorem.
We construct the required new family within the union
of the given cycles and these paths.
Iterating Lemma~\ref{lem:prefix-bounded} then gives a single
cycle while controlling the total decrease in length.
This proves the conclusions of Theorem~\ref{thm:general}.

We now turn to the proof of Theorem~\ref{thm:main}. Let $G$ be a connected
vertex-transitive graph on $n$ vertices.
Recall that a ball of radius $r$ consists of the vertices
at distance at most $r$ from its center; its size $\beta_r$
is independent of the center.
We first establish in Lemma~\ref{lem:ballcap} an inequality valid for
every connected vertex-transitive graph.
This enables us to verify condition~\eqref{eq:component-condition}
with an absolute constant $K$ whenever $\beta_r>r^{5/4}$
for every integer $R_0\le r<n^{4/5}$ with some absolute constant $R_0$.
Applying Theorem~\ref{thm:general} then gives both required
conclusions.
It therefore remains to consider the case where
$\beta_R\le R^{5/4}$ for some integer
$R_0\le R<n^{4/5}$.
Here we apply a theorem of DeVos and Mohar, stated below
as Theorem~\ref{thm:DM}, to obtain a cyclically ordered partition of $V(G)$.
From this partition, we construct a quotient graph that is a Cayley graph of a cyclic or a dihedral group.
The quotient graph has a Hamiltonian cycle in the cyclic case by the classical result of Maru{\v{s}}i{\v{c}}~\cite{Marusic1983},
and Lemma~\ref{lem:dihedral-cycle} gives a cycle covering
almost all vertices in the dihedral case.
Using Lemma~\ref{lem:normal-quantitative}, we lift this cycle
to $G$ with a small loss and obtain a cycle of desired length, which completes the proof of Theorem~\ref{thm:main}.

\paragraph{Organization.}
The rest of the paper is organized as follows.
In Section~\ref{sec:growth}, we prove Theorem~\ref{thm:general} for the general criterion.
In Section~\ref{sec:vertex-transitive}, we prove our main result -- Theorem~\ref{thm:main} for vertex-transitive graphs.
We conclude in Section~\ref{sec:robust} with a discussion of related problems, including an application to regular robust sublinear expanders.

\section{Proof of Theorem~\ref{thm:general}}\label{sec:growth}
Throughout this section, 
we fix $K\ge0$, and let $G$ be a graph on $n$ vertices satisfying~\eqref{eq:component-condition}.
For a family of disjoint cycles, we call the sum of their
lengths the \emph{total length}.
\subsection{Initial families of disjoint cycles}\label{sec:initial-family}

The following lemma yields a family of disjoint cycles of total length linear in $n$, which will be needed in the proof of Theorem~\ref{thm:general}.
Recall that a \textit{DFS tree} of a connected graph is a rooted spanning
tree obtained by depth-first search.
Given a DFS tree, two vertices are \textit{comparable} if one lies on the subpath of the tree between the other and the root.
Every edge of the graph joins comparable vertices in the DFS tree.

\begin{lemma}\label{lem:prefix-initial}
For every integer $\ell\ge3$, there exist $\nu>0$ and an integer $n_0$ such that, if $n\ge n_0$, then $G$ contains disjoint cycles of minimum
length at least $\ell$ and total length at least $\nu n$.
\end{lemma}
\begin{proof}
Choose an integer $\lambda$ such that
$\lambda^{1/8}\ge4\ell(K+1)$.
Let $n$ be sufficiently large so that
$n\ge2\lambda$.
Choose an inclusion-maximal family of disjoint cycles of length
at least $\ell$, and denote the union of their vertex sets by $U$.
Then the maximality of this family implies that $H:=G-U$ has no cycle of length at least $\ell$.

We will apply~\eqref{eq:component-condition} to a set
$S\supseteq U$ such that every component of $G-S$ has at most
$\lambda$ vertices and $|S\setminus U|\le\ell n/\lambda$.
Start with $S=U$, so $G-S=H$.
Whenever $G-S$ has a component with more than $\lambda$
vertices, perform one step as follows.
Choose such a component $C$ and a DFS tree $T$ of $C$,
rooted at any vertex. Among the vertices whose descendant
subtrees have more than $\lambda$ vertices, choose a vertex
$z$ at maximum distance from the root in $T$.
Add to $S$ the vertex $z$ and the at most $\ell-2$ ancestors
nearest to $z$ on its path to the root.
The construction terminates when every component of $G-S$
has at most $\lambda$ vertices.

To prove $|S\setminus U|\le\ell n/\lambda$, we bound the
number of steps. Fix a step and let $T_z$ be the subtree
of $T$ rooted at $z$. Every edge of $C$ joins comparable
vertices in $T$ at distance at most $\ell-2$, as otherwise there would be a cycle of length at least
$\ell$. Consequently, after this step every vertex of
$T_z-z$ lies in a component of $G-S$ of size at most
$\lambda$. As $z$ is added to $S$,
no vertex of $T_z$ can be chosen later.
The subtrees $T_z$ obtained in different steps are therefore
pairwise disjoint in $H$. Since $|T_z|\ge\lambda+1$ by the choice of $z$,
there are at most $|H|/(\lambda+1)$ steps, each adding
at most $\ell-1$ vertices to $S$. Hence
\[
 |S\setminus U|
 \le \frac{(\ell-1)|H|}{\lambda+1}
 \le \frac{\ell n}{\lambda}.
\]

Apply~\eqref{eq:component-condition} to $S$, writing
$X_1,\ldots,X_t$ for the components of $G-S$ in nonincreasing
order of size. Since each has at most $\lambda$ vertices,
\[
\begin{aligned}
 n\le\lambda+|S|+\lambda^{7/8}\sum_{j=2}^t|X_j|^{1/8}\le\lambda+(1+K\lambda^{7/8})|S|
 \le\frac n2+\frac{\lambda}{4\ell}|S|.
\end{aligned}
\]
The last inequality uses $n\ge2\lambda$ and
$1+K\lambda^{7/8}\le(K+1)\lambda^{7/8}\le\lambda/(4\ell)$
by the choice of $\lambda$.
Thus $|S|\ge2\ell n/\lambda$, and
\[
 |U|=|S|-|S\setminus U|
 \ge\frac{2\ell n}{\lambda}-\frac{\ell n}{\lambda}
 =\frac{\ell n}{\lambda}.
\]
The chosen cycles have total length $|U|$, so
$\nu=\ell/\lambda$ suffices.
\end{proof}

For the second conclusion of Theorem~\ref{thm:general},
we need disjoint cycles that together cover almost all vertices and whose minimum length tends to infinity with the degree.
We use the following theorem in its original directed form.
A directed graph is $d$-regular if every vertex has indegree and
outdegree $d$; here loops and directed cycles of length two are
allowed, but parallel arcs are not.

\begin{theorem}[Christoph--Dragani\'c--Gir\~ao--Hurley--Michel--M\"uyesser {\cite[Theorem~1.1]{CDGHMM}}]\label{thm:few-cycles}
Let $D$ be a directed $d$-regular graph on $n$ vertices, where
$d\ge2$. A uniformly random spanning family of disjoint directed
cycles in $D$ has $O(n\log d/d)$ cycles in expectation.
\end{theorem}

\subsection{Increasing the minimum cycle length}\label{sec:replacement}

The key lemma of this section increases the minimum length
in a family of disjoint cycles while preserving almost all
of their total length.

\begin{lemma}\label{lem:prefix-bounded}
There is an integer $x_0$, depending only on $K$, such that
the following holds for every integer
$x\ge x_0$.
If $G$ contains a family of at least two disjoint cycles of
minimum length at least $x$ and total length $M$, then
$G$ contains a family of disjoint cycles of minimum length at
least $3x/2$ and total length at least $(1-13x^{-1/80})M.$
\end{lemma}

For the proof of Lemma~\ref{lem:prefix-bounded}, we need
disjoint paths joining distinct cycles, with all internal
vertices outside the cycles and sufficiently many endpoints
on every cycle.
We obtain these paths in Lemma~\ref{lem:ports} using the
following tools for disjoint vertex sets.

Let $F$ be a graph, and let
$\mathcal T=\{T_1,\ldots,T_N\}$ be a family of disjoint
vertex sets in $F$.
A path in $F$ is called a \emph{$\mathcal T$-path} if its
endpoints lie in different sets $T_i$ and its internal
vertices lie outside $\bigcup_{i=1}^N T_i$.
A \emph{$\mathcal T$-packing} is a family of pairwise
vertex-disjoint $\mathcal T$-paths, and a
\emph{$\mathcal T$-transversal} is a subset of $V(F)$
meeting every $\mathcal T$-path.
We use the following corollary of Mader's packing theorem.

\begin{lemma}\label{thm:mader-transversal}
Let $\mathcal T$ be a family of disjoint vertex sets
in a graph $F$.
There is a $\mathcal T$-transversal of size at most
twice the maximum size of a $\mathcal T$-packing.
\end{lemma}
\begin{proof}
Let $\nu$ be the maximum size of a $\mathcal T$-packing.
Mader's theorem, in the form stated
in~\cite[p.~319]{SchrijverPaths}, gives pairwise disjoint
vertex sets $U_0,B_1,\ldots,B_r$ such that
\[
 \nu=|U_0|+\sum_{j=1}^r
              \left\lfloor\frac{|B_j|}{2}\right\rfloor
\]
and every $\mathcal T$-path avoiding $U_0$ contains
at least two vertices of some $B_j$.
Taking all vertices of $U_0$ and all but one vertex
of each nonempty $B_j$ therefore gives a
$\mathcal T$-transversal of size
\[
 |U_0|+\sum_{j=1}^r\max\{|B_j|-1,0\}
 \le |U_0|+2\sum_{j=1}^r
                 \left\lfloor\frac{|B_j|}{2}\right\rfloor
 \le 2\nu. \qedhere
\]
\end{proof}

We also need a packing with sufficiently many
endpoints in every $T_i$.

\begin{lemma}\label{thm:endpoint-selection}
Let $\mathcal T=\{T_1,\ldots,T_N\}$ be a family of disjoint
vertex sets in a graph $F$, and let $b\ge1$ be an integer.
Suppose that for every $I\subseteq\{1,\ldots,N\}$, some
$\mathcal T$-packing has at least $b|I|$ endpoints in
$\bigcup_{i\in I}T_i$.
Then there is a $\mathcal T$-packing with at least $b$
endpoints in every $T_i$.
\end{lemma}

\begin{proof}
Let $\mathcal I$ consist of the subsets of
$\bigcup_{i=1}^N T_i$ that are contained in the endpoint
set of some $\mathcal T$-packing.
This family contains the empty set and is closed under
taking subsets.
Schrijver~\cite[p.~1]{SchrijverMaderMatroid} proved that
$\mathcal I$ also satisfies the following exchange property:
if $A,B\in\mathcal I$ and $|A|<|B|$, then
$A\cup\{v\}\in\mathcal I$ for some $v\in B\setminus A$.

For a family $\mathcal I$ with these properties,
Rado's theorem~\cite[Corollary~41.1c]{SchrijverMatroids}
gives distinct representatives of a list of sets whose
set belongs to $\mathcal I$, provided that the union
of every sublist contains a member of $\mathcal I$
of size at least the number of entries in that sublist.

Apply this theorem to a list containing $b$ copies
of each $T_i$.
A sublist involving precisely the indices in $I$
has at most $b|I|$ entries.
By the hypothesis, its union $\bigcup_{i\in I}T_i$
contains at least $b|I|$ endpoints of some
$\mathcal T$-packing, and hence a member of
$\mathcal I$ of size at least $b|I|$.
Rado's theorem therefore gives $b$ distinct
representatives from each $T_i$ whose union belongs
to $\mathcal I$.
By the definition of $\mathcal I$, one
$\mathcal T$-packing has all these representatives
among its endpoints.
\end{proof}

The proof of the next lemma follows the argument of
Segal and Verstra\"ete~\cite[Lemma~2.1]{SegalVerstraete2026},
with~\eqref{eq:component-condition} providing the
lower bound on transversal sizes.

\begin{lemma}\label{lem:ports}
Let $b\ge1$ be an integer, and let
$\mathcal T=\{T_1,\ldots,T_N\}$ be a family of disjoint vertex
sets in $G$, where $N\ge2$ and
$|T_i|^{1/8}\ge2(K+1)b$ for every $i$.
Then there is a $\mathcal T$-packing with at least $b$ endpoints
in every $T_i$.
\end{lemma}
\begin{proof}
We first show that every $\mathcal A$-transversal has size
at least $2b(m-1)$, where
$\mathcal A=\{A_1,\ldots,A_m\}$ is any family of disjoint
vertex sets in $G$ with $m\ge2$ and
$|A_i|^{1/8}\ge2(K+1)b$ for every $i$.
Let $Z$ be an $\mathcal A$-transversal, and let $\mathcal C$
consist of the components of $G-Z$, with a largest component
omitted if one exists.
No component of $G-Z$ meets two different sets $A_i$:
otherwise, a shortest path in that component whose endpoints
lie in different sets $A_i$ would be an $\mathcal A$-path
avoiding $Z$.
We may therefore relabel the sets $A_i$ so that every
component meeting $A_1,\ldots,A_{m-1}$ belongs to $\mathcal C$.
By subadditivity,
\[
\begin{aligned}
 2(K+1)b(m-1)
 &\le \sum_{i=1}^{m-1}|A_i|^{1/8}
 \le \sum_{i=1}^{m-1}|A_i\cap Z|^{1/8}
      +\sum_{C\in\mathcal C}|C|^{1/8}\\
 &\le |Z|+K|Z|=(K+1)|Z|.
\end{aligned}
\]
The second inequality uses the fact that each component
meets at most one $A_i$.
For the last inequality, the sets $A_i\cap Z$ are disjoint
and satisfy $|A_i\cap Z|^{1/8}\le |A_i\cap Z|$, while
\eqref{eq:component-condition} bounds the sum over
$\mathcal C$ by $K|Z|$.
Dividing by $K+1$, we obtain
\begin{equation}\label{eq:terminal-transversal}
 |Z|\ge2b(m-1).
\end{equation}

We apply Lemma~\ref{thm:endpoint-selection} to $\mathcal T$
with parameter $b$ by verifying that, for every
$I\subseteq\{1,\ldots,N\}$, there is a $\mathcal T$-packing
with at least $b|I|$ endpoints in $\bigcup_{i\in I}T_i$.
First suppose that $I$ is a nonempty proper subset of $\{1,\ldots,N\}$, and let
\[
 \mathcal T_I
 =\{T_i:i\in I\}\cup
   \left\{\bigcup_{j\notin I}T_j\right\}.
\]
The $|I|+1$ members of $\mathcal T_I$ are disjoint and
each has size at least $[2(K+1)b]^8$.
By~\eqref{eq:terminal-transversal}, every
$\mathcal T_I$-transversal has size at least $2b|I|$.
Lemma~\ref{thm:mader-transversal} therefore gives a
$\mathcal T_I$-packing of at least $b|I|$ paths.
Every $\mathcal T_I$-path has its endpoints in different
sets $T_j$, at least one of them indexed by $I$, and
its internal vertices outside $\bigcup_{j=1}^N T_j$.
Thus this packing is also a $\mathcal T$-packing with
at least $b|I|$ endpoints in $\bigcup_{i\in I}T_i$.

For $I=\{1,\ldots,N\}$, applying
\eqref{eq:terminal-transversal} with $m=N$ and $A_i=T_i$
shows that every $\mathcal T$-transversal has size at
least $2b(N-1)$.
Lemma~\ref{thm:mader-transversal} then gives a
$\mathcal T$-packing of at least $b(N-1)$ paths,
and hence at least $2b(N-1)\ge bN$ endpoints.
The condition for $I=\emptyset$ is trivial, so
Lemma~\ref{thm:endpoint-selection} gives the required packing.
\end{proof}

\begin{proof}[\normalfont\bfseries Proof of Lemma~\ref{lem:prefix-bounded}]
Fix a family $\mathcal H$ of at least two disjoint cycles of
minimum length at least $x$ and total length $M$, and let
$s=\lceil x^{1/80}\rceil$.
We call a cycle in $\mathcal H$ \emph{short} if its length is
less than $2x$, and \emph{long} otherwise.
By choosing $x_0$ sufficiently large, we may assume $s\ge64$
and $x^{1/8}\ge2(K+1)s^8$ for every $x\ge x_0$.
Lemma~\ref{lem:ports}, applied to the vertex sets of the cycles in
$\mathcal H$ with $b=s^8$, gives a family $\mathcal P$ of
disjoint paths, each joining distinct cycles of $\mathcal H$
and having all internal vertices outside every cycle in
$\mathcal H$, with at least $s^8$ endpoints on every cycle.

Our aim is to construct disjoint cycles of length at least
$3x/2$ and total length at least $(1-13/s)M$.
Let $J$ be an auxiliary graph with vertex set $\mathcal H$,
where two cycles are adjacent if at least $s^6$ paths in
$\mathcal P$ join them, and let $\mathcal R$ consist of the
short cycles with no neighbor in $J$.
We first use cycles in a family $\mathcal B\subseteq\mathcal H$
to construct longer cycles using cycles in $\mathcal R\setminus\mathcal B$.
The following claim provides the required $\mathcal B$.

\begin{claim}\label{clm:choose-b}
There is a subfamily $\mathcal B\subseteq\mathcal H$ such that
\begin{equation}\label{eq:b-cost}
 \sum_{D\in\mathcal B}|D|+2x|\mathcal B|\le6M/s.
\end{equation}
Every short cycle outside $\mathcal B\cup\mathcal R$ has a neighbor in $J-\mathcal B$\footnote{Here $\mathcal B$ is a subset of $V(J)$.}, and for every
$D\in\mathcal R\setminus\mathcal B$, at least $s^7/2$ paths
in $\mathcal P$ join $D$ to cycles in $\mathcal B$ in total.
\end{claim}

\begin{proof}[\normalfont\bfseries Proof of Claim~\ref{clm:choose-b}]
For each $D\in\mathcal R$, choose $s^8$ paths in $\mathcal P$
with an endpoint on $D$. For $E\in\mathcal H\setminus\{D\}$,
let $a_{D,E}$ be the number of these $s^8$ chosen paths
that join $D$ to $E$.
Let $\mathcal B_0$ be obtained by selecting each member of
$\mathcal H$ independently with probability $1/s$.
For $D\in\mathcal R$, the number of chosen paths joining
$D$ to cycles in $\mathcal B_0$ is
\[
 X_D=\sum_{E\in\mathcal H\setminus\{D\}}
          a_{D,E}\mathbf 1_{\{E\in\mathcal B_0\}}.
\]
Let $\mathcal B$ be obtained from $\mathcal B_0$ by adding
the following cycles:
\begin{itemize}[leftmargin=*]
\item every $D\in\mathcal R\setminus\mathcal B_0$ with
$X_D<s^7/2$;
\item every short cycle outside $\mathcal R\cup\mathcal B_0$
with no neighbor in $J-\mathcal B_0$.
\end{itemize}
The first addition ensures the required paths connecting
$\mathcal R\setminus\mathcal B$ to $\mathcal B$.
By the second addition, every short cycle outside
$\mathcal B\cup\mathcal R$ has a neighbor in $J-\mathcal B_0$.
Every added cycle is isolated in $J-\mathcal B_0$, so
this neighbor remains in $J-\mathcal B$.
It therefore remains to prove~\eqref{eq:b-cost}.

We estimate the probability that a fixed cycle $D$ is added
to $\mathcal B_0$. The first rule applies only to cycles
in $\mathcal R$, and the second only to short cycles outside
$\mathcal R$, so each short cycle is subject to just one rule.
For $D\in\mathcal R$, we have $a_{D,E}<s^6$ for every $E$
by the definition of $J$, and $\sum_{E\ne D}a_{D,E}=s^8$.
Thus $X_D/s^6$ is a sum of independent random variables in
$[0,1]$ with expectation $s$. Chernoff's inequality gives
\[
 \Pr(D\in\mathcal B\setminus\mathcal B_0)
 \le \Pr(X_D<s^7/2)\le e^{-s/8}\le1/s.
\]
For a short cycle $D\notin\mathcal R$, fix a neighbor $E$
of $D$ in $J$.
The second rule adds $D$ only if all its neighbors lie in
$\mathcal B_0$, which in particular requires $E\in\mathcal B_0$.
Hence
$\Pr(D\in\mathcal B\setminus\mathcal B_0)
 \le\Pr(E\in\mathcal B_0)=1/s$.
Neither rule adds a long cycle, so
$\Pr(D\in\mathcal B\setminus\mathcal B_0)\le1/s$
for every $D\in\mathcal H$.
Since $\Pr(D\in\mathcal B_0)=1/s$, it follows that
$\Pr(D\in\mathcal B)\le2/s$.
By linearity of expectation,
\[
\begin{aligned}
 \mathbb E\left(\sum_{D\in\mathcal B}|D|
                         +2x|\mathcal B|\right)
 &=\sum_{D\in\mathcal H}(|D|+2x)\Pr(D\in\mathcal B)\\
 &\le \frac2s\bigl(M+2x|\mathcal H|\bigr)
 \le \frac{6M}{s},
\end{aligned}
\]
where the last inequality uses $x|\mathcal H|\le M$.
A choice of $\mathcal B$ satisfying~\eqref{eq:b-cost}
therefore exists.
\end{proof}

Fix $\mathcal B$ as in Claim~\ref{clm:choose-b}.
For each $D\in\mathcal R\setminus\mathcal B$, we construct
a path $L_D$ between distinct members of $\mathcal B$
containing at least $(1-2/s)|D|$ vertices of $D$.
In Figure~\ref{fig:rb-construction}, $B_1,\ldots,B_4$ belong to
$\mathcal B$, and the smaller cycles belong to $\mathcal R\setminus\mathcal B$.
We seek two paths from $D$ to distinct cycles in $\mathcal B$
whose endpoints on $D$ bound a short subpath.
To distinguish their destinations, for every path in $\mathcal P$ joining $D$ to a cycle
$E\in\mathcal B$, label its endpoint on $D$ by $E$.
There are at least $s^7/2$ labelled endpoints on $D$,
and each label occurs fewer than $s^6$ times.
In cyclic order along $D$, each maximal sequence of
consecutive endpoints with the same label therefore contains
fewer than $s^6$ endpoints. Hence at least $s/2$ pairs of
consecutive labelled endpoints have different labels.
Since the subpaths between these consecutive pairs have total length at most $|D|$, by the pigeonhole principle, one of these subpaths has length at most $2|D|/s$. The complementary subpath and the two corresponding
paths in $\mathcal P$ form the required path $L_D$;
see Figure~\ref{fig:rb-construction}(a)--(b).
The paths $L_D$ are disjoint and meet the cycles in
$\mathcal B$ only at their endpoints.

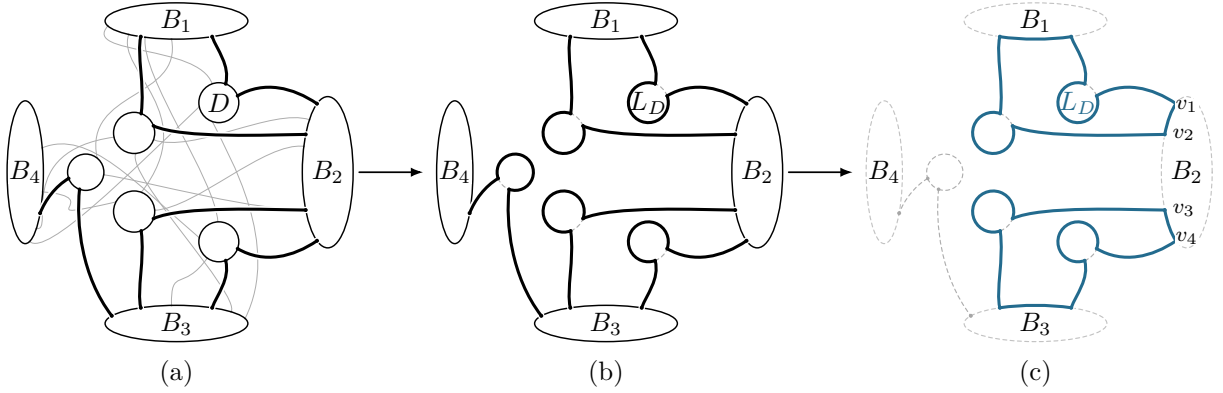
\begin{figure}[htbp]
\centering
\begingroup
\definecolor{cycleblue}{HTML}{236B8E}
\begin{tikzpicture}[x=.80cm,y=.80cm,line cap=round,line join=round,
 every node/.style={font=\small,inner sep=.8pt},
 candidate/.style={draw=black!32,line width=.28pt},
 ordinary/.style={draw=black,line width=.55pt},
 retained/.style={draw=black,line width=1.15pt},
 omitted/.style={draw=black!32,line width=.43pt,dash pattern=on 1.5pt off 1.25pt},
 context/.style={draw=black!24,line width=.43pt,dash pattern=on 1.5pt off 1.25pt}]
\path[use as bounding box] (-3.08,-3.35) rectangle (17.28,3.02);

\begin{scope}[shift={(0.00,0)}]
\draw[candidate] (0.59035,1.46125) .. controls (0.35775,2.12147) and (-1.56551,1.93095) .. (-1.13429,2.41731);
\draw[candidate] (0.59334,0.83771) .. controls (0.87604,0.25373) and (1.74065,-1.19994) .. (1.13429,-2.41731);
\draw[candidate] (0.37694,1.08270) .. controls (-0.77960,-0.01888) and (-2.14498,-1.40186) .. (-2.42742,-1.14495);
\draw[candidate] (-0.37947,0.53661) .. controls (0.28046,0.30315) and (1.57961,1.06893) .. (2.19788,0.86832);
\draw[candidate] (-0.67315,0.31106) .. controls (-0.83853,-0.05828) and (1.06533,-2.01400) .. (0.90393,-2.30716);
\draw[candidate] (-1.03275,0.58015) .. controls (-2.71459,0.34487) and (-1.99740,-0.99680) .. (-2.31530,-0.92985);
\draw[candidate] (-0.58077,-0.33159) .. controls (0.43871,0.98644) and (-1.00626,1.66826) .. (-0.78957,2.27706);
\draw[candidate] (-0.36063,-0.62935) .. controls (0.33808,-0.58684) and (1.46014,0.50643) .. (2.10533,0.42753);
\draw[candidate] (-1.00060,-0.49114) .. controls (-2.56570,-0.84663) and (-1.09026,-0.41592) .. (-2.21162,-0.32525);
\draw[candidate] (0.68356,-0.82041) .. controls (0.64867,-0.12128) and (-0.59727,1.58530) .. (-0.51728,2.23036);
\draw[candidate] (0.50880,-1.41897) .. controls (0.10323,-1.98950) and (-0.09387,-1.55083) .. (-0.08231,-2.20073);
\draw[candidate] (0.37694,-1.08270) .. controls (-0.66570,-0.19535) and (-1.44017,0.72235) .. (-2.21809,0.40358);
\draw[candidate] (-1.31135,0.23326) .. controls (-1.54473,1.37012) and (0.10767,1.19594) .. (-0.08231,2.20073);
\draw[candidate] (-1.20292,-0.04175) .. controls (-0.50973,-0.13917) and (1.48929,-0.70115) .. (2.12916,-0.58684);
\fill[black!36] (0.59035,1.46125) circle[radius=.017];
\fill[black!36] (-1.13429,2.41731) circle[radius=.017];
\fill[black!36] (0.59334,0.83771) circle[radius=.017];
\fill[black!36] (1.13429,-2.41731) circle[radius=.017];
\fill[black!36] (0.37694,1.08270) circle[radius=.017];
\fill[black!36] (-2.42742,-1.14495) circle[radius=.017];
\fill[black!36] (-0.37947,0.53661) circle[radius=.017];
\fill[black!36] (2.19788,0.86832) circle[radius=.017];
\fill[black!36] (-0.67315,0.31106) circle[radius=.017];
\fill[black!36] (0.90393,-2.30716) circle[radius=.017];
\fill[black!36] (-1.03275,0.58015) circle[radius=.017];
\fill[black!36] (-2.31530,-0.92985) circle[radius=.017];
\fill[black!36] (-0.58077,-0.33159) circle[radius=.017];
\fill[black!36] (-0.78957,2.27706) circle[radius=.017];
\fill[black!36] (-0.36063,-0.62935) circle[radius=.017];
\fill[black!36] (2.10533,0.42753) circle[radius=.017];
\fill[black!36] (-1.00060,-0.49114) circle[radius=.017];
\fill[black!36] (-2.21162,-0.32525) circle[radius=.017];
\fill[black!36] (0.68356,-0.82041) circle[radius=.017];
\fill[black!36] (-0.51728,2.23036) circle[radius=.017];
\fill[black!36] (0.50880,-1.41897) circle[radius=.017];
\fill[black!36] (-0.08231,-2.20073) circle[radius=.017];
\fill[black!36] (0.37694,-1.08270) circle[radius=.017];
\fill[black!36] (-2.21809,0.40358) circle[radius=.017];
\fill[black!36] (-1.31135,0.23326) circle[radius=.017];
\fill[black!36] (-0.08231,2.20073) circle[radius=.017];
\fill[black!36] (-1.20292,-0.04175) circle[radius=.017];
\fill[black!36] (2.12916,-0.58684) circle[radius=.017];
\draw[ordinary] (0,2.5) ellipse[x radius=1.18,y radius=0.3];
\draw[ordinary] (2.5,0) ellipse[x radius=0.42,y radius=1.25];
\draw[ordinary] (0,-2.5) ellipse[x radius=1.18,y radius=0.3];
\draw[ordinary] (-2.5,0) ellipse[x radius=0.3,y radius=1.18];
\draw[draw=white,line width=2.45pt] (0.81287,1.46010) .. controls (0.89000,1.77000) and (0.67000,2.02000) .. (0.59000,2.24019);
\draw[retained] (0.81287,1.46010) .. controls (0.89000,1.77000) and (0.67000,2.02000) .. (0.59000,2.24019);
\draw[draw=white,line width=2.45pt] (1.01010,1.26287) .. controls (1.45000,1.52000) and (1.93000,1.40000) .. (2.32250,1.13288);
\draw[retained] (1.01010,1.26287) .. controls (1.45000,1.52000) and (1.93000,1.40000) .. (2.32250,1.13288);
\draw[draw=white,line width=2.45pt] (-0.58371,0.96950) .. controls (-0.52000,1.44000) and (-0.66000,1.95000) .. (-0.59000,2.24019);
\draw[retained] (-0.58371,0.96950) .. controls (-0.52000,1.44000) and (-0.66000,1.95000) .. (-0.59000,2.24019);
\draw[draw=white,line width=2.45pt] (-0.38050,0.76629) .. controls (0.10000,0.55000) and (1.54000,0.62000) .. (2.13627,0.62500);
\draw[retained] (-0.38050,0.76629) .. controls (0.10000,0.55000) and (1.54000,0.62000) .. (2.13627,0.62500);
\draw[draw=white,line width=2.45pt] (-0.58371,-0.96950) .. controls (-0.52000,-1.44000) and (-0.66000,-1.95000) .. (-0.59000,-2.24019);
\draw[retained] (-0.58371,-0.96950) .. controls (-0.52000,-1.44000) and (-0.66000,-1.95000) .. (-0.59000,-2.24019);
\draw[draw=white,line width=2.45pt] (-0.38050,-0.76629) .. controls (0.10000,-0.55000) and (1.54000,-0.62000) .. (2.13627,-0.62500);
\draw[retained] (-0.38050,-0.76629) .. controls (0.10000,-0.55000) and (1.54000,-0.62000) .. (2.13627,-0.62500);
\draw[draw=white,line width=2.45pt] (0.81287,-1.46010) .. controls (0.89000,-1.77000) and (0.67000,-2.02000) .. (0.59000,-2.24019);
\draw[retained] (0.81287,-1.46010) .. controls (0.89000,-1.77000) and (0.67000,-2.02000) .. (0.59000,-2.24019);
\draw[draw=white,line width=2.45pt] (1.01010,-1.26287) .. controls (1.45000,-1.52000) and (1.93000,-1.40000) .. (2.32250,-1.13288);
\draw[retained] (1.01010,-1.26287) .. controls (1.45000,-1.52000) and (1.93000,-1.40000) .. (2.32250,-1.13288);
\draw[draw=white,line width=2.45pt] (-1.78191,-0.10261) .. controls (-2.00000,-0.23000) and (-2.17000,-0.46000) .. (-2.25425,-0.67682);
\draw[retained] (-1.78191,-0.10261) .. controls (-2.00000,-0.23000) and (-2.17000,-0.46000) .. (-2.25425,-0.67682);
\draw[draw=white,line width=2.45pt] (-1.60261,-0.28191) .. controls (-1.66000,-0.85000) and (-1.57000,-1.72000) .. (-1.06944,-2.37321);
\draw[retained] (-1.60261,-0.28191) .. controls (-1.66000,-0.85000) and (-1.57000,-1.72000) .. (-1.06944,-2.37321);
\fill[white] (0.7,1.15) circle[radius=0.33];
\draw[ordinary] (0.7,1.15) circle[radius=0.33];
\node[text=black,font=\small] at (0.7,1.15) {$D$};
\fill[black] (1.01010,1.26287) circle[radius=.030];
\fill[black] (0.81287,1.46010) circle[radius=.030];
\fill[white] (-0.7,0.65) circle[radius=0.34];
\draw[ordinary] (-0.7,0.65) circle[radius=0.34];
\fill[black] (-0.38050,0.76629) circle[radius=.030];
\fill[black] (-0.58371,0.96950) circle[radius=.030];
\fill[white] (-0.7,-0.65) circle[radius=0.34];
\draw[ordinary] (-0.7,-0.65) circle[radius=0.34];
\fill[black] (-0.58371,-0.96950) circle[radius=.030];
\fill[black] (-0.38050,-0.76629) circle[radius=.030];
\fill[white] (0.7,-1.15) circle[radius=0.33];
\draw[ordinary] (0.7,-1.15) circle[radius=0.33];
\fill[black] (0.81287,-1.46010) circle[radius=.030];
\fill[black] (1.01010,-1.26287) circle[radius=.030];
\fill[white] (-1.5,0) circle[radius=0.3];
\draw[ordinary] (-1.5,0) circle[radius=0.3];
\fill[black] (-1.78191,-0.10261) circle[radius=.030];
\fill[black] (-1.60261,-0.28191) circle[radius=.030];
\fill[black] (0.59000,2.24019) circle[radius=.033];
\fill[black] (2.32250,1.13288) circle[radius=.033];
\fill[black] (-0.59000,2.24019) circle[radius=.033];
\fill[black] (2.13627,0.62500) circle[radius=.033];
\fill[black] (-0.59000,-2.24019) circle[radius=.033];
\fill[black] (2.13627,-0.62500) circle[radius=.033];
\fill[black] (0.59000,-2.24019) circle[radius=.033];
\fill[black] (2.32250,-1.13288) circle[radius=.033];
\fill[black] (-2.25425,-0.67682) circle[radius=.033];
\fill[black] (-1.06944,-2.37321) circle[radius=.033];
\node[font=\small,fill=white,inner sep=.3pt] at (0,2.5) {$B_{1}$};
\node[font=\small,fill=white,inner sep=.3pt] at (2.5,0) {$B_{2}$};
\node[font=\small,fill=white,inner sep=.3pt] at (0,-2.5) {$B_{3}$};
\node[font=\small,fill=white,inner sep=.3pt] at (-2.5,0) {$B_{4}$};
\node[font=\small] at (0,-3.33) {(a)};
\end{scope}
\begin{scope}[shift={(7.10,0)}]
\draw[ordinary] (0,2.5) ellipse[x radius=1.18,y radius=0.3];
\draw[ordinary] (2.5,0) ellipse[x radius=0.42,y radius=1.25];
\draw[ordinary] (0,-2.5) ellipse[x radius=1.18,y radius=0.3];
\draw[ordinary] (-2.5,0) ellipse[x radius=0.3,y radius=1.18];
\draw[draw=white,line width=2.45pt] (0.81287,1.46010) .. controls (0.89000,1.77000) and (0.67000,2.02000) .. (0.59000,2.24019);
\draw[retained] (0.81287,1.46010) .. controls (0.89000,1.77000) and (0.67000,2.02000) .. (0.59000,2.24019);
\draw[draw=white,line width=2.45pt] (1.01010,1.26287) .. controls (1.45000,1.52000) and (1.93000,1.40000) .. (2.32250,1.13288);
\draw[retained] (1.01010,1.26287) .. controls (1.45000,1.52000) and (1.93000,1.40000) .. (2.32250,1.13288);
\draw[draw=white,line width=2.45pt] (-0.58371,0.96950) .. controls (-0.52000,1.44000) and (-0.66000,1.95000) .. (-0.59000,2.24019);
\draw[retained] (-0.58371,0.96950) .. controls (-0.52000,1.44000) and (-0.66000,1.95000) .. (-0.59000,2.24019);
\draw[draw=white,line width=2.45pt] (-0.38050,0.76629) .. controls (0.10000,0.55000) and (1.54000,0.62000) .. (2.13627,0.62500);
\draw[retained] (-0.38050,0.76629) .. controls (0.10000,0.55000) and (1.54000,0.62000) .. (2.13627,0.62500);
\draw[draw=white,line width=2.45pt] (-0.58371,-0.96950) .. controls (-0.52000,-1.44000) and (-0.66000,-1.95000) .. (-0.59000,-2.24019);
\draw[retained] (-0.58371,-0.96950) .. controls (-0.52000,-1.44000) and (-0.66000,-1.95000) .. (-0.59000,-2.24019);
\draw[draw=white,line width=2.45pt] (-0.38050,-0.76629) .. controls (0.10000,-0.55000) and (1.54000,-0.62000) .. (2.13627,-0.62500);
\draw[retained] (-0.38050,-0.76629) .. controls (0.10000,-0.55000) and (1.54000,-0.62000) .. (2.13627,-0.62500);
\draw[draw=white,line width=2.45pt] (0.81287,-1.46010) .. controls (0.89000,-1.77000) and (0.67000,-2.02000) .. (0.59000,-2.24019);
\draw[retained] (0.81287,-1.46010) .. controls (0.89000,-1.77000) and (0.67000,-2.02000) .. (0.59000,-2.24019);
\draw[draw=white,line width=2.45pt] (1.01010,-1.26287) .. controls (1.45000,-1.52000) and (1.93000,-1.40000) .. (2.32250,-1.13288);
\draw[retained] (1.01010,-1.26287) .. controls (1.45000,-1.52000) and (1.93000,-1.40000) .. (2.32250,-1.13288);
\draw[draw=white,line width=2.45pt] (-1.78191,-0.10261) .. controls (-2.00000,-0.23000) and (-2.17000,-0.46000) .. (-2.25425,-0.67682);
\draw[retained] (-1.78191,-0.10261) .. controls (-2.00000,-0.23000) and (-2.17000,-0.46000) .. (-2.25425,-0.67682);
\draw[draw=white,line width=2.45pt] (-1.60261,-0.28191) .. controls (-1.66000,-0.85000) and (-1.57000,-1.72000) .. (-1.06944,-2.37321);
\draw[retained] (-1.60261,-0.28191) .. controls (-1.66000,-0.85000) and (-1.57000,-1.72000) .. (-1.06944,-2.37321);
\fill[white] (0.7,1.15) circle[radius=0.33];
\draw[retained] (0.81287,1.46010) arc[start angle=70,end angle=380,radius=0.33];
\draw[omitted] (1.01010,1.26287) arc[start angle=20,end angle=70,radius=0.33];
\node[text=black,font=\small] at (0.7,1.15) {$L_D$};
\fill[black] (1.01010,1.26287) circle[radius=.030];
\fill[black] (0.81287,1.46010) circle[radius=.030];
\fill[white] (-0.7,0.65) circle[radius=0.34];
\draw[retained] (-0.58371,0.96950) arc[start angle=70,end angle=380,radius=0.34];
\draw[omitted] (-0.38050,0.76629) arc[start angle=20,end angle=70,radius=0.34];
\fill[black] (-0.38050,0.76629) circle[radius=.030];
\fill[black] (-0.58371,0.96950) circle[radius=.030];
\fill[white] (-0.7,-0.65) circle[radius=0.34];
\draw[retained] (-0.38050,-0.76629) arc[start angle=340,end angle=650,radius=0.34];
\draw[omitted] (-0.58371,-0.96950) arc[start angle=290,end angle=340,radius=0.34];
\fill[black] (-0.58371,-0.96950) circle[radius=.030];
\fill[black] (-0.38050,-0.76629) circle[radius=.030];
\fill[white] (0.7,-1.15) circle[radius=0.33];
\draw[retained] (1.01010,-1.26287) arc[start angle=340,end angle=650,radius=0.33];
\draw[omitted] (0.81287,-1.46010) arc[start angle=290,end angle=340,radius=0.33];
\fill[black] (0.81287,-1.46010) circle[radius=.030];
\fill[black] (1.01010,-1.26287) circle[radius=.030];
\fill[white] (-1.5,0) circle[radius=0.3];
\draw[retained] (-1.60261,-0.28191) arc[start angle=250,end angle=560,radius=0.3];
\draw[omitted] (-1.78191,-0.10261) arc[start angle=200,end angle=250,radius=0.3];
\fill[black] (-1.78191,-0.10261) circle[radius=.030];
\fill[black] (-1.60261,-0.28191) circle[radius=.030];
\fill[black] (0.59000,2.24019) circle[radius=.033];
\fill[black] (2.32250,1.13288) circle[radius=.033];
\fill[black] (-0.59000,2.24019) circle[radius=.033];
\fill[black] (2.13627,0.62500) circle[radius=.033];
\fill[black] (-0.59000,-2.24019) circle[radius=.033];
\fill[black] (2.13627,-0.62500) circle[radius=.033];
\fill[black] (0.59000,-2.24019) circle[radius=.033];
\fill[black] (2.32250,-1.13288) circle[radius=.033];
\fill[black] (-2.25425,-0.67682) circle[radius=.033];
\fill[black] (-1.06944,-2.37321) circle[radius=.033];
\node[font=\small,fill=white,inner sep=.3pt] at (0,2.5) {$B_{1}$};
\node[font=\small,fill=white,inner sep=.3pt] at (2.5,0) {$B_{2}$};
\node[font=\small,fill=white,inner sep=.3pt] at (0,-2.5) {$B_{3}$};
\node[font=\small,fill=white,inner sep=.3pt] at (-2.5,0) {$B_{4}$};
\node[font=\small] at (0,-3.33) {(b)};
\end{scope}
\begin{scope}[shift={(14.20,0)}]
\draw[context] (0,2.5) ellipse[x radius=1.18,y radius=0.3];
\draw[context] (2.5,0) ellipse[x radius=0.42,y radius=1.25];
\draw[context] (0,-2.5) ellipse[x radius=1.18,y radius=0.3];
\draw[context] (-2.5,0) ellipse[x radius=0.3,y radius=1.18];
\draw[draw=white,line width=2.45pt] (0.81287,1.46010) .. controls (0.89000,1.77000) and (0.67000,2.02000) .. (0.59000,2.24019);
\draw[retained,draw=cycleblue] (0.81287,1.46010) .. controls (0.89000,1.77000) and (0.67000,2.02000) .. (0.59000,2.24019);
\draw[draw=white,line width=2.45pt] (1.01010,1.26287) .. controls (1.45000,1.52000) and (1.93000,1.40000) .. (2.32250,1.13288);
\draw[retained,draw=cycleblue] (1.01010,1.26287) .. controls (1.45000,1.52000) and (1.93000,1.40000) .. (2.32250,1.13288);
\draw[draw=white,line width=2.45pt] (-0.58371,0.96950) .. controls (-0.52000,1.44000) and (-0.66000,1.95000) .. (-0.59000,2.24019);
\draw[retained,draw=cycleblue] (-0.58371,0.96950) .. controls (-0.52000,1.44000) and (-0.66000,1.95000) .. (-0.59000,2.24019);
\draw[draw=white,line width=2.45pt] (-0.38050,0.76629) .. controls (0.10000,0.55000) and (1.54000,0.62000) .. (2.13627,0.62500);
\draw[retained,draw=cycleblue] (-0.38050,0.76629) .. controls (0.10000,0.55000) and (1.54000,0.62000) .. (2.13627,0.62500);
\draw[draw=white,line width=2.45pt] (-0.58371,-0.96950) .. controls (-0.52000,-1.44000) and (-0.66000,-1.95000) .. (-0.59000,-2.24019);
\draw[retained,draw=cycleblue] (-0.58371,-0.96950) .. controls (-0.52000,-1.44000) and (-0.66000,-1.95000) .. (-0.59000,-2.24019);
\draw[draw=white,line width=2.45pt] (-0.38050,-0.76629) .. controls (0.10000,-0.55000) and (1.54000,-0.62000) .. (2.13627,-0.62500);
\draw[retained,draw=cycleblue] (-0.38050,-0.76629) .. controls (0.10000,-0.55000) and (1.54000,-0.62000) .. (2.13627,-0.62500);
\draw[draw=white,line width=2.45pt] (0.81287,-1.46010) .. controls (0.89000,-1.77000) and (0.67000,-2.02000) .. (0.59000,-2.24019);
\draw[retained,draw=cycleblue] (0.81287,-1.46010) .. controls (0.89000,-1.77000) and (0.67000,-2.02000) .. (0.59000,-2.24019);
\draw[draw=white,line width=2.45pt] (1.01010,-1.26287) .. controls (1.45000,-1.52000) and (1.93000,-1.40000) .. (2.32250,-1.13288);
\draw[retained,draw=cycleblue] (1.01010,-1.26287) .. controls (1.45000,-1.52000) and (1.93000,-1.40000) .. (2.32250,-1.13288);
\draw[omitted] (-1.78191,-0.10261) .. controls (-2.00000,-0.23000) and (-2.17000,-0.46000) .. (-2.25425,-0.67682);
\draw[omitted] (-1.60261,-0.28191) .. controls (-1.66000,-0.85000) and (-1.57000,-1.72000) .. (-1.06944,-2.37321);
\fill[white] (0.7,1.15) circle[radius=0.33];
\draw[retained,draw=cycleblue] (0.81287,1.46010) arc[start angle=70,end angle=380,radius=0.33];
\draw[omitted] (1.01010,1.26287) arc[start angle=20,end angle=70,radius=0.33];
\node[text=cycleblue,font=\small] at (0.7,1.15) {$L_D$};
\fill[cycleblue] (1.01010,1.26287) circle[radius=.030];
\fill[cycleblue] (0.81287,1.46010) circle[radius=.030];
\fill[white] (-0.7,0.65) circle[radius=0.34];
\draw[retained,draw=cycleblue] (-0.58371,0.96950) arc[start angle=70,end angle=380,radius=0.34];
\draw[omitted] (-0.38050,0.76629) arc[start angle=20,end angle=70,radius=0.34];
\fill[cycleblue] (-0.38050,0.76629) circle[radius=.030];
\fill[cycleblue] (-0.58371,0.96950) circle[radius=.030];
\fill[white] (-0.7,-0.65) circle[radius=0.34];
\draw[retained,draw=cycleblue] (-0.38050,-0.76629) arc[start angle=340,end angle=650,radius=0.34];
\draw[omitted] (-0.58371,-0.96950) arc[start angle=290,end angle=340,radius=0.34];
\fill[cycleblue] (-0.58371,-0.96950) circle[radius=.030];
\fill[cycleblue] (-0.38050,-0.76629) circle[radius=.030];
\fill[white] (0.7,-1.15) circle[radius=0.33];
\draw[retained,draw=cycleblue] (1.01010,-1.26287) arc[start angle=340,end angle=650,radius=0.33];
\draw[omitted] (0.81287,-1.46010) arc[start angle=290,end angle=340,radius=0.33];
\fill[cycleblue] (0.81287,-1.46010) circle[radius=.030];
\fill[cycleblue] (1.01010,-1.26287) circle[radius=.030];
\fill[white] (-1.5,0) circle[radius=0.3];
\draw[omitted] (-1.60261,-0.28191) arc[start angle=250,end angle=560,radius=0.3];
\draw[omitted] (-1.78191,-0.10261) arc[start angle=200,end angle=250,radius=0.3];
\fill[black!35] (-1.78191,-0.10261) circle[radius=.030];
\fill[black!35] (-1.60261,-0.28191) circle[radius=.030];
\draw[retained,draw=cycleblue] (-0.59000,2.24019) arc[start angle=240,end angle=300,x radius=1.18,y radius=0.3];
\draw[retained,draw=cycleblue] (2.32250,1.13288) arc[start angle=115,end angle=150,x radius=0.42,y radius=1.25];
\draw[retained,draw=cycleblue] (2.13627,-0.62500) arc[start angle=210,end angle=245,x radius=0.42,y radius=1.25];
\draw[retained,draw=cycleblue] (0.59000,-2.24019) arc[start angle=60,end angle=120,x radius=1.18,y radius=0.3];
\fill[cycleblue] (0.59000,2.24019) circle[radius=.033];
\fill[cycleblue] (2.32250,1.13288) circle[radius=.033];
\fill[cycleblue] (-0.59000,2.24019) circle[radius=.033];
\fill[cycleblue] (2.13627,0.62500) circle[radius=.033];
\fill[cycleblue] (-0.59000,-2.24019) circle[radius=.033];
\fill[cycleblue] (2.13627,-0.62500) circle[radius=.033];
\fill[cycleblue] (0.59000,-2.24019) circle[radius=.033];
\fill[cycleblue] (2.32250,-1.13288) circle[radius=.033];
\fill[black!35] (-2.25425,-0.67682) circle[radius=.033];
\fill[black!35] (-1.06944,-2.37321) circle[radius=.033];
\node[font=\small,fill=white,inner sep=.3pt] at (0,2.5) {$B_{1}$};
\node[font=\small,fill=white,inner sep=.3pt] at (2.5,0) {$B_{2}$};
\node[font=\small,fill=white,inner sep=.3pt] at (0,-2.5) {$B_{3}$};
\node[font=\small,fill=white,inner sep=.3pt] at (-2.5,0) {$B_{4}$};
\node[font=\scriptsize,fill=white,inner sep=.3pt] at (2.5,1.09) {$v_{1}$};
\node[font=\scriptsize,fill=white,inner sep=.3pt] at (2.45,0.625) {$v_{2}$};
\node[font=\scriptsize,fill=white,inner sep=.3pt] at (2.45,-0.625) {$v_{3}$};
\node[font=\scriptsize,fill=white,inner sep=.3pt] at (2.5,-1.09) {$v_{4}$};
\node[font=\small] at (0,-3.33) {(c)};
\end{scope}
\draw[-{Latex[length=1.65mm,width=1.25mm]},line width=.65pt] (3.02,0)--(4.08,0);
\draw[-{Latex[length=1.65mm,width=1.25mm]},line width=.65pt] (10.12,0)--(11.18,0);
\end{tikzpicture}
\endgroup
\caption{Constructing disjoint cycles from $\mathcal B\cup\mathcal R$.}
\label{fig:rb-construction}
\end{figure}

We now construct disjoint cycles of length at least $3x/2$
from the paths $L_D$ and subpaths of cycles in $\mathcal B$.
Let $F$ be an auxiliary multigraph with vertex set $\mathcal B$.
The edges of $F$ are defined as follows: for each path $L_D$,
add one edge between the two members of $\mathcal B$
containing the endpoints of $L_D$.
Choose an inclusion-maximal spanning subgraph $F'$ of $F$
in which every degree is even.
By maximality, the edges in $E(F)\setminus E(F')$ contain no cycle, and hence $|E(F')|\geq |E(F)|-|\mathcal B|$.
Let $\mathcal L$ be the family of paths $L_D$ corresponding
to the edges of $F'$.
For each $C\in\mathcal B$, list the endpoints on $C$
of paths in $\mathcal L$ as $v_1,\ldots,v_{2r}$ in cyclic
order, where $C$ has degree $2r$ in $F'$.
For $1\le j\le r$, take the subpath of $C$ from
$v_{2j-1}$ to $v_{2j}$ following this cyclic order.
The union of these subpaths over all $C\in\mathcal B$
and the paths in $\mathcal L$ has degree two at every
vertex, so it is a family of disjoint cycles;
see Figure~\ref{fig:rb-construction}(c).
Each resulting cycle contains at least two members
of $\mathcal L$, since every path in $\mathcal L$
joins distinct members of $\mathcal B$.
Its length is therefore at least $2(1-2/s)x\ge3x/2$.
We now bound the number of vertices on cycles in
$\mathcal B\cup\mathcal R$ not covered by the constructed cycles.
At most $|\mathcal B|$ cycles $D\in\mathcal R\setminus\mathcal B$
satisfy $L_D\notin\mathcal L$, and each has length less than $2x$.
Hence the number of uncovered vertices on these cycles $D$
and on cycles in $\mathcal B$ is at most
$\sum_{D\in\mathcal B}|D|+2x|\mathcal B|\le6M/s$
by~\eqref{eq:b-cost}.
For each $D\in\mathcal R\setminus\mathcal B$ with
$L_D\in\mathcal L$, at most $2|D|/s$ vertices of $D$
are omitted in constructing $L_D$.
Thus the number of uncovered vertices on cycles in
$\mathcal B\cup\mathcal R$ is at most
\[
 \frac{6M}{s}
 +\frac2s\sum_{D\in\mathcal R\setminus\mathcal B}|D|.
\]
This completes the construction of disjoint cycles of
length at least $3x/2$ from $\mathcal B\cup\mathcal R$.

\medskip

We then consider the short cycles outside $\mathcal B\cup\mathcal R$.
By Claim~\ref{clm:choose-b}, every such cycle has a neighbor in
$J-(\mathcal B\cup\mathcal R)$.
We may therefore choose disjoint stars of order at least two
in this graph that cover all these short cycles and whose
leaves are short cycles.
Fix one star, with center $Q$ and leaves $D_1,\ldots,D_t$\footnote{All the leaves are short cycles, while the center is possibly long.},
where $t\ge1$.
Write $M_Q=|Q|+\sum_{i=1}^t|D_i|$ for the total length
of the cycles in this star.
For each $i$, let $\mathcal P_i$ consist of exactly $s^6$
paths in $\mathcal P$ from $Q$ to $D_i$.
After deleting one edge of $Q$, we partition the resulting
path into a family $\mathcal Q$ of $t$ disjoint subpaths,
each containing exactly $s^6$ endpoints of the paths in
$\bigcup_{i=1}^t\mathcal P_i$.
These subpaths cover all vertices of $Q$.

Let $\mathcal D$ consist of the cycles $D_i$ for which at least $s^4$
paths in $\mathcal P_i$ end on one member of $\mathcal Q$,
and for each $D_i\in\mathcal D$, fix such a member $Q_i$.
For $P\in\mathcal Q$, let $\mathcal D_P=\{D_i\in\mathcal D:Q_i=P\}$.
Since $P$ contains $s^6$ endpoints in total, $|\mathcal D_P|\le s^2$.
The next claim inserts the cycles in $\mathcal D_P$.

\begin{claim}\label{clm:insert-path}
For every $P\in\mathcal Q$, the union of $P$, the cycles
in $\mathcal D_P$, and the paths in
$\bigcup_{D_i\in\mathcal D_P}\mathcal P_i$ ending on $P$ contains
a path $P'$ with the same endpoints as $P$ and at least
a $(1-2/s)$ fraction of the vertices of
$P\cup\big(\bigcup_{D\in\mathcal D_P}D\big)$.
\end{claim}

\begin{proof}[\normalfont\bfseries Proof of Claim~\ref{clm:insert-path}]
If $\mathcal D_P=\emptyset$, take $P'=P$.
Otherwise, order the vertices of $P$ from one endpoint to the other.
We first choose disjoint subpaths $I_i$ of $P$, one for each
$D_i\in\mathcal D_P$, such that $I_i$ has at most $2|P|/s^3$ vertices
and contains the endpoints of $s$ paths in $\mathcal P_i$;
see Figure~\ref{fig:interval-insertion}.

\usetikzlibrary{decorations.pathreplacing}
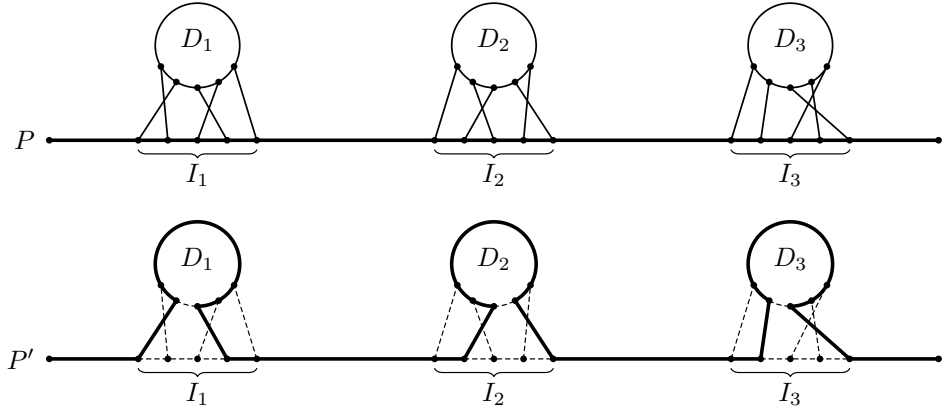
\begin{figure}[htbp]
\centering
\begingroup
\begin{tikzpicture}[x=.98cm,y=.98cm,line cap=round,line join=round,
 every node/.style={font=\small,inner sep=1pt},
 original/.style={line width=.65pt},
 retained/.style={line width=1.35pt},
 discarded/.style={line width=.45pt,dash pattern=on 1.7pt off 1.4pt}]

\begin{scope}[shift={(0,0)}]
\draw[retained] (0,0)--(12,0);
\node[anchor=east] at (-.15,0) {$P$};
\draw[original] (2,1.28) circle[radius=.57];
\node at (2,1.35) {$D_1$};
\draw[original] (1.2,0)--({2+.57*cos(240)},{1.28+.57*sin(240)});
\draw[original] (1.6,0)--({2+.57*cos(210)},{1.28+.57*sin(210)});
\draw[original] (2,0)--({2+.57*cos(300)},{1.28+.57*sin(300)});
\draw[original] (2.4,0)--({2+.57*cos(270)},{1.28+.57*sin(270)});
\draw[original] (2.8,0)--({2+.57*cos(330)},{1.28+.57*sin(330)});
\fill (1.2,0) circle[radius=1.25pt];
\fill ({2+.57*cos(240)},{1.28+.57*sin(240)}) circle[radius=1.25pt];
\fill (1.6,0) circle[radius=1.25pt];
\fill ({2+.57*cos(210)},{1.28+.57*sin(210)}) circle[radius=1.25pt];
\fill (2,0) circle[radius=1.25pt];
\fill ({2+.57*cos(300)},{1.28+.57*sin(300)}) circle[radius=1.25pt];
\fill (2.4,0) circle[radius=1.25pt];
\fill ({2+.57*cos(270)},{1.28+.57*sin(270)}) circle[radius=1.25pt];
\fill (2.8,0) circle[radius=1.25pt];
\fill ({2+.57*cos(330)},{1.28+.57*sin(330)}) circle[radius=1.25pt];
\draw[decorate,decoration={brace,mirror,amplitude=3pt},line width=.45pt] (1.2,-.13)--(2.8,-.13) node[midway,below=4pt] {$I_1$};
\draw[original] (6,1.28) circle[radius=.57];
\node at (6,1.35) {$D_2$};
\draw[original] (5.2,0)--({6+.57*cos(210)},{1.28+.57*sin(210)});
\draw[original] (5.6,0)--({6+.57*cos(270)},{1.28+.57*sin(270)});
\draw[original] (6,0)--({6+.57*cos(240)},{1.28+.57*sin(240)});
\draw[original] (6.4,0)--({6+.57*cos(330)},{1.28+.57*sin(330)});
\draw[original] (6.8,0)--({6+.57*cos(300)},{1.28+.57*sin(300)});
\fill (5.2,0) circle[radius=1.25pt];
\fill ({6+.57*cos(210)},{1.28+.57*sin(210)}) circle[radius=1.25pt];
\fill (5.6,0) circle[radius=1.25pt];
\fill ({6+.57*cos(270)},{1.28+.57*sin(270)}) circle[radius=1.25pt];
\fill (6,0) circle[radius=1.25pt];
\fill ({6+.57*cos(240)},{1.28+.57*sin(240)}) circle[radius=1.25pt];
\fill (6.4,0) circle[radius=1.25pt];
\fill ({6+.57*cos(330)},{1.28+.57*sin(330)}) circle[radius=1.25pt];
\fill (6.8,0) circle[radius=1.25pt];
\fill ({6+.57*cos(300)},{1.28+.57*sin(300)}) circle[radius=1.25pt];
\draw[decorate,decoration={brace,mirror,amplitude=3pt},line width=.45pt] (5.2,-.13)--(6.8,-.13) node[midway,below=4pt] {$I_2$};
\draw[original] (10,1.28) circle[radius=.57];
\node at (10,1.35) {$D_3$};
\draw[original] (9.2,0)--({10+.57*cos(210)},{1.28+.57*sin(210)});
\draw[original] (9.6,0)--({10+.57*cos(240)},{1.28+.57*sin(240)});
\draw[original] (10,0)--({10+.57*cos(330)},{1.28+.57*sin(330)});
\draw[original] (10.4,0)--({10+.57*cos(300)},{1.28+.57*sin(300)});
\draw[original] (10.8,0)--({10+.57*cos(270)},{1.28+.57*sin(270)});
\fill (9.2,0) circle[radius=1.25pt];
\fill ({10+.57*cos(210)},{1.28+.57*sin(210)}) circle[radius=1.25pt];
\fill (9.6,0) circle[radius=1.25pt];
\fill ({10+.57*cos(240)},{1.28+.57*sin(240)}) circle[radius=1.25pt];
\fill (10,0) circle[radius=1.25pt];
\fill ({10+.57*cos(330)},{1.28+.57*sin(330)}) circle[radius=1.25pt];
\fill (10.4,0) circle[radius=1.25pt];
\fill ({10+.57*cos(300)},{1.28+.57*sin(300)}) circle[radius=1.25pt];
\fill (10.8,0) circle[radius=1.25pt];
\fill ({10+.57*cos(270)},{1.28+.57*sin(270)}) circle[radius=1.25pt];
\draw[decorate,decoration={brace,mirror,amplitude=3pt},line width=.45pt] (9.2,-.13)--(10.8,-.13) node[midway,below=4pt] {$I_3$};
\fill (0,0) circle[radius=1.25pt]; \fill (12,0) circle[radius=1.25pt];
\end{scope}
\begin{scope}[shift={(0,-2.95)}]
\draw[retained] (0,0)--(1.2,0);
\draw[discarded] (1.2,0)--(2.4,0);
\draw[retained] (2.4,0)--(5.6,0);
\draw[discarded] (5.6,0)--(6.8,0);
\draw[retained] (6.8,0)--(9.6,0);
\draw[discarded] (9.6,0)--(10.8,0);
\draw[retained] (10.8,0)--(12,0);
\node[anchor=east] at (-.15,0) {$P'$};
\draw[retained] ({2+.57*cos(240)},{1.28+.57*sin(240)}) arc[start angle=240,end angle=-90,radius=.57];
\draw[discarded] ({2+.57*cos(240)},{1.28+.57*sin(240)}) arc[start angle=240,end angle=270,radius=.57];
\node at (2,1.35) {$D_1$};
\draw[retained] (1.2,0)--({2+.57*cos(240)},{1.28+.57*sin(240)});
\draw[discarded] (1.6,0)--({2+.57*cos(210)},{1.28+.57*sin(210)});
\draw[discarded] (2,0)--({2+.57*cos(300)},{1.28+.57*sin(300)});
\draw[retained] (2.4,0)--({2+.57*cos(270)},{1.28+.57*sin(270)});
\draw[discarded] (2.8,0)--({2+.57*cos(330)},{1.28+.57*sin(330)});
\fill (1.2,0) circle[radius=1.25pt];
\fill ({2+.57*cos(240)},{1.28+.57*sin(240)}) circle[radius=1.25pt];
\fill (1.6,0) circle[radius=1.25pt];
\fill ({2+.57*cos(210)},{1.28+.57*sin(210)}) circle[radius=1.25pt];
\fill (2,0) circle[radius=1.25pt];
\fill ({2+.57*cos(300)},{1.28+.57*sin(300)}) circle[radius=1.25pt];
\fill (2.4,0) circle[radius=1.25pt];
\fill ({2+.57*cos(270)},{1.28+.57*sin(270)}) circle[radius=1.25pt];
\fill (2.8,0) circle[radius=1.25pt];
\fill ({2+.57*cos(330)},{1.28+.57*sin(330)}) circle[radius=1.25pt];
\draw[decorate,decoration={brace,mirror,amplitude=3pt},line width=.45pt] (1.2,-.13)--(2.8,-.13) node[midway,below=4pt] {$I_1$};
\draw[retained] ({6+.57*cos(270)},{1.28+.57*sin(270)}) arc[start angle=270,end angle=-60,radius=.57];
\draw[discarded] ({6+.57*cos(270)},{1.28+.57*sin(270)}) arc[start angle=270,end angle=300,radius=.57];
\node at (6,1.35) {$D_2$};
\draw[discarded] (5.2,0)--({6+.57*cos(210)},{1.28+.57*sin(210)});
\draw[retained] (5.6,0)--({6+.57*cos(270)},{1.28+.57*sin(270)});
\draw[discarded] (6,0)--({6+.57*cos(240)},{1.28+.57*sin(240)});
\draw[discarded] (6.4,0)--({6+.57*cos(330)},{1.28+.57*sin(330)});
\draw[retained] (6.8,0)--({6+.57*cos(300)},{1.28+.57*sin(300)});
\fill (5.2,0) circle[radius=1.25pt];
\fill ({6+.57*cos(210)},{1.28+.57*sin(210)}) circle[radius=1.25pt];
\fill (5.6,0) circle[radius=1.25pt];
\fill ({6+.57*cos(270)},{1.28+.57*sin(270)}) circle[radius=1.25pt];
\fill (6,0) circle[radius=1.25pt];
\fill ({6+.57*cos(240)},{1.28+.57*sin(240)}) circle[radius=1.25pt];
\fill (6.4,0) circle[radius=1.25pt];
\fill ({6+.57*cos(330)},{1.28+.57*sin(330)}) circle[radius=1.25pt];
\fill (6.8,0) circle[radius=1.25pt];
\fill ({6+.57*cos(300)},{1.28+.57*sin(300)}) circle[radius=1.25pt];
\draw[decorate,decoration={brace,mirror,amplitude=3pt},line width=.45pt] (5.2,-.13)--(6.8,-.13) node[midway,below=4pt] {$I_2$};
\draw[retained] ({10+.57*cos(240)},{1.28+.57*sin(240)}) arc[start angle=240,end angle=-90,radius=.57];
\draw[discarded] ({10+.57*cos(240)},{1.28+.57*sin(240)}) arc[start angle=240,end angle=270,radius=.57];
\node at (10,1.35) {$D_3$};
\draw[discarded] (9.2,0)--({10+.57*cos(210)},{1.28+.57*sin(210)});
\draw[retained] (9.6,0)--({10+.57*cos(240)},{1.28+.57*sin(240)});
\draw[discarded] (10,0)--({10+.57*cos(330)},{1.28+.57*sin(330)});
\draw[discarded] (10.4,0)--({10+.57*cos(300)},{1.28+.57*sin(300)});
\draw[retained] (10.8,0)--({10+.57*cos(270)},{1.28+.57*sin(270)});
\fill (9.2,0) circle[radius=1.25pt];
\fill ({10+.57*cos(210)},{1.28+.57*sin(210)}) circle[radius=1.25pt];
\fill (9.6,0) circle[radius=1.25pt];
\fill ({10+.57*cos(240)},{1.28+.57*sin(240)}) circle[radius=1.25pt];
\fill (10,0) circle[radius=1.25pt];
\fill ({10+.57*cos(330)},{1.28+.57*sin(330)}) circle[radius=1.25pt];
\fill (10.4,0) circle[radius=1.25pt];
\fill ({10+.57*cos(300)},{1.28+.57*sin(300)}) circle[radius=1.25pt];
\fill (10.8,0) circle[radius=1.25pt];
\fill ({10+.57*cos(270)},{1.28+.57*sin(270)}) circle[radius=1.25pt];
\draw[decorate,decoration={brace,mirror,amplitude=3pt},line width=.45pt] (9.2,-.13)--(10.8,-.13) node[midway,below=4pt] {$I_3$};
\fill (0,0) circle[radius=1.25pt]; \fill (12,0) circle[radius=1.25pt];
\end{scope}
\end{tikzpicture}
\endgroup
\caption{Inserting $D_i$ into $P$ within disjoint intervals $I_i$ to obtain a path $P'$.}
\label{fig:interval-insertion}
\end{figure}

For each $D_i\in\mathcal D_P$, choose $s^4$ paths from $\mathcal P_i$
ending on $P$, and divide their endpoints along $P$ into
$s^3$ consecutive sets of $s$ endpoints.
The $s^3$ subpaths of $P$ from the first to the last endpoint
of each set are pairwise disjoint, so their vertex counts
sum to at most $|P|$.
Let $\mathcal I_i$ consist of those subpaths $I$ for which $|V(I)|\le2|P|/s^3$.
It follows that $|\mathcal I_i|\ge s^3/2$.

Select one subpath from each family in stages
$k=1,\ldots,|\mathcal D_P|$.
At stage $k$, among all subpaths in the current families,
choose one whose last vertex occurs first along $P$.
Write $i_k$ for the index of its family and $I_{i_k}$ for
the chosen subpath.
We then remove the family $\mathcal I_{i_k}$ and delete
every subpath meeting $I_{i_k}$ from each remaining family.
A subpath deleted from another family must contain the
last vertex of $I_{i_k}$, since it meets $I_{i_k}$ and
cannot end before that vertex.
The members of each family are disjoint, so this update
deletes at most one subpath from each remaining family.
Each family initially has at least $s^3/2\ge s^2\ge |\mathcal D_P|$
members. Before stage $k$, each remaining family has lost
at most $k-1$ members and is therefore nonempty.
Thus every stage is possible, and the construction selects
one $I_i$ for every $D_i\in\mathcal D_P$, with all selected subpaths disjoint.

We construct $P'$ by replacing a subpath of each $I_i$
with a path through $D_i$. For $D_i\in\mathcal D_P$, the $s$ chosen
paths ending on $I_i$ have distinct endpoints on $D_i$.
By the pigeonhole principle, some two consecutive endpoints
on $D_i$ bound a subpath of length at most $|D_i|/s$.
Replacing the subpath of $P$ between their corresponding
endpoints by the two paths to $D_i$ and the complementary
subpath of $D_i$ inserts all but at most $|D_i|/s$ vertices
of $D_i$ into $P$.
The replacements occur within the disjoint subpaths $I_i$,
so together they form a path $P'$ with the same endpoints as $P$.
They lose at most $s^2\cdot2|P|/s^3=2|P|/s$ vertices of $P$
and at most $|D_i|/s$ vertices of each inserted cycle,
proving the claim.
\end{proof}

We now choose a subfamily $\mathcal S\subseteq\mathcal Q$.
Once $\mathcal S$ is fixed, we will proceed as follows:
\begin{itemize}
\item[(1)] Construct paths from cycles $D_i\notin\mathcal D$
joining distinct members of $\mathcal S$.
\item[(2)] Replace each $P\in\mathcal Q\setminus\mathcal S$ by
the path $P'$ given by Claim~\ref{clm:insert-path}.
Keeping every edge of these paths $P'$ and most of the
paths constructed in (1), delete subpaths only within
members of $\mathcal S$ to obtain a family of disjoint cycles.
\item[(3)] Discard the cycles of length less than $3x/2$.
\end{itemize}

For a choice of $\mathcal S$, let $\mathcal E$ consist of
the following cycles:
\begin{itemize}[leftmargin=*]
\item the cycles $D_i\in\mathcal D$ with $Q_i\in\mathcal S$;
\item the cycles $D_i\notin\mathcal D$ for which fewer than
$s^5/2$ paths in $\mathcal P_i$ end on members of $\mathcal S$.
\end{itemize}
The new cycles constructed for this star will contain
no vertices of the cycles in $\mathcal E$.
We choose $\mathcal S$ so that
\begin{equation}\label{eq:star-selection}
 \sum_{P\in\mathcal S}|P|+\sum_{D\in\mathcal E}|D|
       +4x|\mathcal S|\le\frac{5M_Q}{s}.
\end{equation}
Indeed, let $\mathcal S$ be obtained by selecting each member of
$\mathcal Q$ independently with probability $1/s$.
For $D_i\in\mathcal D$, the probability that $D_i\in\mathcal E$ is $1/s$.
For $D_i\notin\mathcal D$, let $N_i$ count the paths in $\mathcal P_i$
ending on members of $\mathcal S$, and write $a_{i,P}$ for
the number ending on $P\in\mathcal Q$. Then
\[
 N_i=\sum_{P\in\mathcal Q}a_{i,P}\mathbf 1_{\{P\in\mathcal S\}},
 \qquad 0\le a_{i,P}<s^4,
 \qquad \sum_{P\in\mathcal Q}a_{i,P}=s^6.
\]
Thus $N_i/s^4$ is a sum of independent random variables
in $[0,1]$ with expectation $s$.
Chernoff's inequality gives
\[
 \Pr(D_i\in\mathcal E)=\Pr(N_i<s^5/2)
 \le e^{-s/8}\le1/s.
\]
Each cycle $D_i$ therefore belongs to $\mathcal E$ with
probability at most $1/s$.
Since each member of $\mathcal Q$ is selected with
probability $1/s$, the linearity of expectation implies that the expectation of the left-hand side in~\eqref{eq:star-selection} is at most
\[
 \frac{|Q|+\sum_{i=1}^t|D_i|+4tx}{s}
 \le\frac{5M_Q}{s},
\]
where the last inequality uses $tx\le\sum_{i=1}^t|D_i|$.
We may therefore fix $\mathcal S$
satisfying~\eqref{eq:star-selection}.

Let $Q'$ be obtained from $Q$ by replacing each
$P\in\mathcal Q\setminus\mathcal S$ with the path $P'$
given by Claim~\ref{clm:insert-path}.
The paths $P'$ are pairwise disjoint, have the same endpoints as the original paths they replace, and meet no member of $\mathcal S$.
Thus $Q'$ is a cycle containing every member of $\mathcal S$ as a subpath.

\medskip

We now construct, for each cycle
$D_i\notin\mathcal D\cup\mathcal E$, a path between distinct
members of $\mathcal S$, as described in the following claim.

\begin{claim}\label{clm:between-parts}
For each $D_i\notin\mathcal D\cup\mathcal E$, there is
a path $L_i$ between distinct members of $\mathcal S$
consisting of two paths in $\mathcal P_i$ and a subpath of $D_i$ containing at least $(1-2/s)|D_i|$ vertices.
\end{claim}

\begin{proof}[\normalfont\bfseries Proof of Claim~\ref{clm:between-parts}]
For each $P\in\mathcal S$, label by $P$ the endpoints
on $D_i$ of paths in $\mathcal P_i$ ending on $P$.
There are at least $s^5/2$ labelled endpoints because
$D_i\notin\mathcal E$, and each label occurs fewer than
$s^4$ times because $D_i\notin\mathcal D$.
In cyclic order along $D_i$, each maximal sequence of
consecutive endpoints with the same label therefore contains
fewer than $s^4$ endpoints. Hence at least $s/2$ pairs of
consecutive labelled endpoints have different labels.
The subpaths between these pairs, containing no other
labelled endpoint, have total length at most $|D_i|$.
By the pigeonhole principle, one has length at most $2|D_i|/s$.
Its complementary subpath in $D_i$ and the two corresponding
paths in $\mathcal P_i$ form the required path $L_i$.

Figure~\ref{fig:cycles-to-paths}(a)--(b) illustrates these replacements.
In (a)--(c), the red subpaths are members of $\mathcal S$.
The black subpaths $P_1'$ and $P_2'$, with
$P_1,P_2\in\mathcal Q\setminus\mathcal S$, have already
replaced $P_1$ and $P_2$ in the construction of $Q'$; here they are
redrawn along the circle representing $Q'$.
\end{proof}

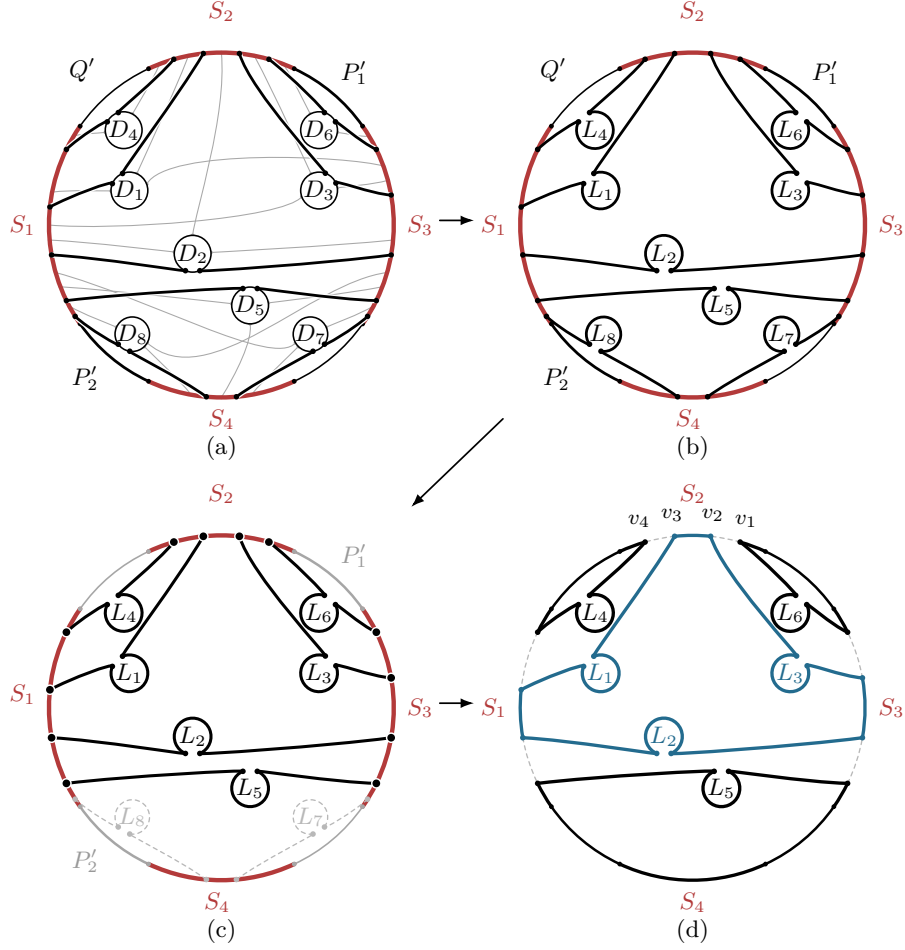
\begin{figure}[!htb]
\centering
\begingroup
\definecolor{cyclered}{HTML}{B33A3A}
\definecolor{cycleblue}{HTML}{236B8E}
\newcommand{\Qarc}[3]{\draw[#1] ({3.0*cos(#2)},{3.0*sin(#2)}) arc[start angle=#2,end angle=#3,radius=3.0];}
\begin{tikzpicture}[x=.76cm,y=.76cm,line cap=round,line join=round,
 every node/.style={font=\footnotesize,inner sep=1pt},
 ordinary/.style={draw=black,line width=.60pt},
 retained/.style={draw=black,line width=1.05pt},
 candidate/.style={draw=black!34,line width=.40pt},
 omitted/.style={draw=black!28,line width=.5pt,dash pattern=on 1.6pt off 1.5pt},
 selected/.style={draw=cyclered,line width=1.65pt}]
\path[use as bounding box] (-3.66,-12.50) rectangle (11.86,4.05);

\begin{scope}[shift={(0,0)}]
\draw[candidate] (-2.94488,0.57243) .. controls (-2.59712,0.61456) and (-2.09989,0.58975) .. (-1.91993,0.59344);
\fill[black!34] (-2.94488,0.57243) circle[radius=.7pt];
\fill[black!34] (-1.91993,0.59344) circle[radius=.7pt];
\draw[candidate] (-0.57243,2.94488) .. controls (-0.84608,2.23323) and (-1.39931,1.05796) .. (-1.47156,0.89309);
\fill[black!34] (-0.57243,2.94488) circle[radius=.7pt];
\fill[black!34] (-1.47156,0.89309) circle[radius=.7pt];
\draw[candidate] (-2.98858,-0.26147) .. controls (-2.24743,-0.29735) and (-0.99772,-0.45229) .. (-0.81854,-0.46947);
\fill[black!34] (-2.98858,-0.26147) circle[radius=.7pt];
\fill[black!34] (-0.81854,-0.46947) circle[radius=.7pt];
\draw[candidate] (2.98858,-0.26147) .. controls (1.91340,-0.37006) and (-0.00116,-0.46589) .. (-0.18075,-0.47817);
\fill[black!34] (2.98858,-0.26147) circle[radius=.7pt];
\fill[black!34] (-0.18075,-0.47817) circle[radius=.7pt];
\draw[candidate] (0.57243,2.94488) .. controls (0.94019,2.26084) and (1.48332,1.05061) .. (1.56132,0.88839);
\fill[black!34] (0.57243,2.94488) circle[radius=.7pt];
\fill[black!34] (1.56132,0.88839) circle[radius=.7pt];
\draw[candidate] (2.89778,0.77646) .. controls (2.60327,0.69769) and (2.19466,0.67287) .. (2.01658,0.64664);
\fill[black!34] (2.89778,0.77646) circle[radius=.7pt];
\fill[black!34] (2.01658,0.64664) circle[radius=.7pt];
\draw[candidate] (-2.57150,1.54511) .. controls (-2.38739,1.60252) and (-2.19642,1.59026) .. (-2.01771,1.61176);
\fill[black!34] (-2.57150,1.54511) circle[radius=.7pt];
\fill[black!34] (-2.01771,1.61176) circle[radius=.7pt];
\draw[candidate] (-1.07510,2.80074) .. controls (-1.20489,2.48840) and (-1.46139,2.08939) .. (-1.54729,1.93121);
\fill[black!34] (-1.07510,2.80074) circle[radius=.7pt];
\fill[black!34] (-1.54729,1.93121) circle[radius=.7pt];
\draw[candidate] (-2.80074,-1.07510) .. controls (-1.78334,-1.14008) and (0.00240,-1.35102) .. (0.18154,-1.36865);
\fill[black!34] (-2.80074,-1.07510) circle[radius=.7pt];
\fill[black!34] (0.18154,-1.36865) circle[radius=.7pt];
\draw[candidate] (2.80074,-1.07510) .. controls (2.13111,-1.20501) and (0.99509,-1.33009) .. (0.81686,-1.35526);
\fill[black!34] (2.80074,-1.07510) circle[radius=.7pt];
\fill[black!34] (0.81686,-1.35526) circle[radius=.7pt];
\draw[candidate] (1.07510,2.80074) .. controls (1.26641,2.52180) and (1.46139,2.08939) .. (1.54729,1.93121);
\fill[black!34] (1.07510,2.80074) circle[radius=.7pt];
\fill[black!34] (1.54729,1.93121) circle[radius=.7pt];
\draw[candidate] (2.57150,1.54511) .. controls (2.37903,1.53303) and (2.19642,1.59026) .. (2.01771,1.61176);
\fill[black!34] (2.57150,1.54511) circle[radius=.7pt];
\fill[black!34] (2.01771,1.61176) circle[radius=.7pt];
\draw[candidate] (2.64884,-1.40841) .. controls (2.38264,-1.56585) and (1.98815,-1.70399) .. (1.82385,-1.77749);
\fill[black!34] (2.64884,-1.40841) circle[radius=.7pt];
\fill[black!34] (1.82385,-1.77749) circle[radius=.7pt];
\draw[candidate] (0.52094,-2.95442) .. controls (0.77453,-2.64447) and (1.21475,-2.24352) .. (1.34047,-2.11470);
\fill[black!34] (0.52094,-2.95442) circle[radius=.7pt];
\fill[black!34] (1.34047,-2.11470) circle[radius=.7pt];
\draw[candidate] (-2.64884,-1.40841) .. controls (-2.38264,-1.56585) and (-1.98815,-1.70399) .. (-1.82385,-1.77749);
\fill[black!34] (-2.64884,-1.40841) circle[radius=.7pt];
\fill[black!34] (-1.82385,-1.77749) circle[radius=.7pt];
\draw[candidate] (-0.52094,-2.95442) .. controls (-0.77453,-2.64447) and (-1.21475,-2.24352) .. (-1.34047,-2.11470);
\fill[black!34] (-0.52094,-2.95442) circle[radius=.7pt];
\fill[black!34] (-1.34047,-2.11470) circle[radius=.7pt];
\draw[candidate] (2.81908,1.02606) .. controls (2.20000,1.20000) and (-0.95000,1.35000) .. (-1.29930,0.70945);
\fill[black!34] (2.81908,1.02606) circle[radius=.7pt];
\fill[black!34] (-1.29930,0.70945) circle[radius=.7pt];
\draw[candidate] (0.00000,3.00000) .. controls (0.10000,1.90000) and (-0.45000,0.25000) .. (-0.50000,-0.18000);
\fill[black!34] (0.00000,3.00000) circle[radius=.7pt];
\fill[black!34] (-0.50000,-0.18000) circle[radius=.7pt];
\draw[candidate] (-3.00000,0.00000) .. controls (-1.90000,-0.10000) and (1.05000,0.10000) .. (1.39930,0.49055);
\fill[black!34] (-3.00000,0.00000) circle[radius=.7pt];
\fill[black!34] (1.39930,0.49055) circle[radius=.7pt];
\draw[candidate] (-0.00000,-3.00000) .. controls (0.25000,-2.60000) and (0.55000,-2.00000) .. (0.50000,-1.72000);
\fill[black!34] (-0.00000,-3.00000) circle[radius=.7pt];
\fill[black!34] (0.50000,-1.72000) circle[radius=.7pt];
\draw[candidate] (-2.88379,-0.82691) .. controls (-1.30000,-1.25000) and (0.90000,-2.50000) .. (1.29019,-2.05000);
\fill[black!34] (-2.88379,-0.82691) circle[radius=.7pt];
\fill[black!34] (1.29019,-2.05000) circle[radius=.7pt];
\draw[candidate] (2.88379,-0.82691) .. controls (1.55301,-0.99415) and (-0.22726,-2.98763) .. (-1.29019,-2.05000);
\fill[black!34] (2.88379,-0.82691) circle[radius=.7pt];
\fill[black!34] (-1.29019,-2.05000) circle[radius=.7pt];
\Qarc{retained}{35}{65}
\Qarc{ordinary}{115}{145}
\Qarc{retained}{215}{245}
\Qarc{ordinary}{295}{325}
\Qarc{selected}{145}{215}
\Qarc{selected}{65}{115}
\Qarc{selected}{325}{395}
\Qarc{selected}{245}{295}
\fill[black] (2.45746,1.72073) circle[radius=.9pt];
\fill[black] (1.26785,2.71892) circle[radius=.9pt];
\fill[black] (-1.26785,2.71892) circle[radius=.9pt];
\fill[black] (-2.45746,1.72073) circle[radius=.9pt];
\fill[black] (-2.45746,-1.72073) circle[radius=.9pt];
\fill[black] (-1.26785,-2.71892) circle[radius=.9pt];
\fill[black] (1.26785,-2.71892) circle[radius=.9pt];
\fill[black] (2.45746,-1.72073) circle[radius=.9pt];
\node[text=cyclered] at (-3.46000,0.00000) {$S_1$};
\node[text=cyclered] at (0.00000,3.72000) {$S_2$};
\node[text=cyclered] at (3.46000,0.00000) {$S_3$};
\node[text=cyclered] at (0.00000,-3.38000) {$S_4$};
\node[] at (2.31,2.65) {$P_1'$};
\node[] at (-2.37,-2.66) {$P_2'$};
\node[] at (-2.42,2.70) {$Q'$};
\draw[draw=white,line width=2.0pt] (-2.98357,0.31359) .. controls (-2.63329,0.50324) and (-1.86860,0.78941) .. (-1.89670,0.71987);
\draw[retained] (-2.98357,0.31359) .. controls (-2.63329,0.50324) and (-1.86860,0.78941) .. (-1.89670,0.71987);
\draw[draw=white,line width=2.0pt] (-0.31359,2.98357) .. controls (-0.74611,2.24330) and (-1.78941,0.86860) .. (-1.71987,0.89670);
\draw[retained] (-0.31359,2.98357) .. controls (-0.74611,2.24330) and (-1.78941,0.86860) .. (-1.71987,0.89670);
\fill[white] (-1.60000,0.60000) circle[radius=0.32];
\draw[ordinary] (-1.60000,0.60000) circle[radius=0.32];
\fill[black] (-2.98357,0.31359) circle[radius=1.05pt];
\fill[black] (-0.31359,2.98357) circle[radius=1.05pt];
\fill[black] (-1.89670,0.71987) circle[radius=1.05pt];
\fill[black] (-1.71987,0.89670) circle[radius=1.05pt];
\node[text=black] at (-1.60000,0.60000) {$D_{1}$};
\draw[draw=white,line width=2.0pt] (-2.95442,-0.52094) .. controls (-2.15601,-0.55935) and (-0.55600,-0.82387) .. (-0.62503,-0.79456);
\draw[retained] (-2.95442,-0.52094) .. controls (-2.15601,-0.55935) and (-0.55600,-0.82387) .. (-0.62503,-0.79456);
\draw[draw=white,line width=2.0pt] (2.95442,-0.52094) .. controls (1.82694,-0.66879) and (-0.44400,-0.82387) .. (-0.37497,-0.79456);
\draw[retained] (2.95442,-0.52094) .. controls (1.82694,-0.66879) and (-0.44400,-0.82387) .. (-0.37497,-0.79456);
\fill[white] (-0.50000,-0.50000) circle[radius=0.32];
\draw[ordinary] (-0.50000,-0.50000) circle[radius=0.32];
\fill[black] (-2.95442,-0.52094) circle[radius=1.05pt];
\fill[black] (2.95442,-0.52094) circle[radius=1.05pt];
\fill[black] (-0.62503,-0.79456) circle[radius=1.05pt];
\fill[black] (-0.37497,-0.79456) circle[radius=1.05pt];
\node[text=black] at (-0.50000,-0.50000) {$D_{2}$};
\draw[draw=white,line width=2.0pt] (0.31359,2.98357) .. controls (0.51685,2.30747) and (1.87992,0.87505) .. (1.80945,0.90070);
\draw[retained] (0.31359,2.98357) .. controls (0.51685,2.30747) and (1.87992,0.87505) .. (1.80945,0.90070);
\draw[draw=white,line width=2.0pt] (2.95442,0.52094) .. controls (2.61563,0.53833) and (1.96183,0.79867) .. (1.99233,0.73016);
\draw[retained] (2.95442,0.52094) .. controls (2.61563,0.53833) and (1.96183,0.79867) .. (1.99233,0.73016);
\fill[white] (1.70000,0.60000) circle[radius=0.32];
\draw[ordinary] (1.70000,0.60000) circle[radius=0.32];
\fill[black] (0.31359,2.98357) circle[radius=1.05pt];
\fill[black] (2.95442,0.52094) circle[radius=1.05pt];
\fill[black] (1.80945,0.90070) circle[radius=1.05pt];
\fill[black] (1.99233,0.73016) circle[radius=1.05pt];
\node[text=black] at (1.70000,0.60000) {$D_{3}$};
\draw[draw=white,line width=2.0pt] (-2.69638,1.31511) .. controls (-2.48533,1.52395) and (-1.95107,1.86210) .. (-1.98512,1.79528);
\draw[retained] (-2.69638,1.31511) .. controls (-2.48533,1.52395) and (-1.95107,1.86210) .. (-1.98512,1.79528);
\draw[draw=white,line width=2.0pt] (-0.82691,2.88379) .. controls (-1.11749,2.52866) and (-1.86528,1.93409) .. (-1.79356,1.95602);
\draw[retained] (-0.82691,2.88379) .. controls (-1.11749,2.52866) and (-1.86528,1.93409) .. (-1.79356,1.95602);
\fill[white] (-1.70000,1.65000) circle[radius=0.32];
\draw[ordinary] (-1.70000,1.65000) circle[radius=0.32];
\fill[black] (-2.69638,1.31511) circle[radius=1.05pt];
\fill[black] (-0.82691,2.88379) circle[radius=1.05pt];
\fill[black] (-1.98512,1.79528) circle[radius=1.05pt];
\fill[black] (-1.79356,1.95602) circle[radius=1.05pt];
\node[text=black] at (-1.70000,1.65000) {$D_{4}$};
\draw[draw=white,line width=2.0pt] (-2.69638,-1.31511) .. controls (-1.65587,-1.18895) and (0.44400,-1.07613) .. (0.37497,-1.10544);
\draw[retained] (-2.69638,-1.31511) .. controls (-1.65587,-1.18895) and (0.44400,-1.07613) .. (0.37497,-1.10544);
\draw[draw=white,line width=2.0pt] (2.69638,-1.31511) .. controls (1.98658,-1.29854) and (0.55600,-1.07613) .. (0.62503,-1.10544);
\draw[retained] (2.69638,-1.31511) .. controls (1.98658,-1.29854) and (0.55600,-1.07613) .. (0.62503,-1.10544);
\fill[white] (0.50000,-1.40000) circle[radius=0.32];
\draw[ordinary] (0.50000,-1.40000) circle[radius=0.32];
\fill[black] (-2.69638,-1.31511) circle[radius=1.05pt];
\fill[black] (2.69638,-1.31511) circle[radius=1.05pt];
\fill[black] (0.37497,-1.10544) circle[radius=1.05pt];
\fill[black] (0.62503,-1.10544) circle[radius=1.05pt];
\node[text=black] at (0.50000,-1.40000) {$D_{5}$};
\draw[draw=white,line width=2.0pt] (0.82691,2.88379) .. controls (1.19366,2.60802) and (1.86528,1.93409) .. (1.79356,1.95602);
\draw[retained] (0.82691,2.88379) .. controls (1.19366,2.60802) and (1.86528,1.93409) .. (1.79356,1.95602);
\draw[draw=white,line width=2.0pt] (2.69638,1.31511) .. controls (2.42378,1.43278) and (1.95107,1.86210) .. (1.98512,1.79528);
\draw[retained] (2.69638,1.31511) .. controls (2.42378,1.43278) and (1.95107,1.86210) .. (1.98512,1.79528);
\fill[white] (1.70000,1.65000) circle[radius=0.32];
\draw[ordinary] (1.70000,1.65000) circle[radius=0.32];
\fill[black] (0.82691,2.88379) circle[radius=1.05pt];
\fill[black] (2.69638,1.31511) circle[radius=1.05pt];
\fill[black] (1.79356,1.95602) circle[radius=1.05pt];
\fill[black] (1.98512,1.79528) circle[radius=1.05pt];
\node[text=black] at (1.70000,1.65000) {$D_{6}$};
\draw[draw=white,line width=2.0pt] (2.54414,-1.58976) .. controls (2.31902,-1.80067) and (1.75068,-2.13527) .. (1.79423,-2.07421);
\draw[retained] (2.54414,-1.58976) .. controls (2.31902,-1.80067) and (1.75068,-2.13527) .. (1.79423,-2.07421);
\draw[draw=white,line width=2.0pt] (0.26147,-2.98858) .. controls (0.68655,-2.67206) and (1.66852,-2.18562) .. (1.59434,-2.19670);
\draw[retained] (0.26147,-2.98858) .. controls (0.68655,-2.67206) and (1.66852,-2.18562) .. (1.59434,-2.19670);
\fill[white] (1.55000,-1.90000) circle[radius=0.3];
\draw[ordinary] (1.55000,-1.90000) circle[radius=0.3];
\fill[black] (2.54414,-1.58976) circle[radius=1.05pt];
\fill[black] (0.26147,-2.98858) circle[radius=1.05pt];
\fill[black] (1.79423,-2.07421) circle[radius=1.05pt];
\fill[black] (1.59434,-2.19670) circle[radius=1.05pt];
\node[text=black] at (1.55000,-1.90000) {$D_{7}$};
\draw[draw=white,line width=2.0pt] (-2.54414,-1.58976) .. controls (-2.31902,-1.80067) and (-1.75068,-2.13527) .. (-1.79423,-2.07421);
\draw[retained] (-2.54414,-1.58976) .. controls (-2.31902,-1.80067) and (-1.75068,-2.13527) .. (-1.79423,-2.07421);
\draw[draw=white,line width=2.0pt] (-0.26147,-2.98858) .. controls (-0.68655,-2.67206) and (-1.66852,-2.18562) .. (-1.59434,-2.19670);
\draw[retained] (-0.26147,-2.98858) .. controls (-0.68655,-2.67206) and (-1.66852,-2.18562) .. (-1.59434,-2.19670);
\fill[white] (-1.55000,-1.90000) circle[radius=0.3];
\draw[ordinary] (-1.55000,-1.90000) circle[radius=0.3];
\fill[black] (-2.54414,-1.58976) circle[radius=1.05pt];
\fill[black] (-0.26147,-2.98858) circle[radius=1.05pt];
\fill[black] (-1.79423,-2.07421) circle[radius=1.05pt];
\fill[black] (-1.59434,-2.19670) circle[radius=1.05pt];
\node[text=black] at (-1.55000,-1.90000) {$D_{8}$};
\node at (0,-3.88) {(a)};
\end{scope}
\begin{scope}[shift={(8.2,0)}]
\Qarc{retained}{35}{65}
\Qarc{ordinary}{115}{145}
\Qarc{retained}{215}{245}
\Qarc{ordinary}{295}{325}
\Qarc{selected}{145}{215}
\Qarc{selected}{65}{115}
\Qarc{selected}{325}{395}
\Qarc{selected}{245}{295}
\fill[black] (2.45746,1.72073) circle[radius=.9pt];
\fill[black] (1.26785,2.71892) circle[radius=.9pt];
\fill[black] (-1.26785,2.71892) circle[radius=.9pt];
\fill[black] (-2.45746,1.72073) circle[radius=.9pt];
\fill[black] (-2.45746,-1.72073) circle[radius=.9pt];
\fill[black] (-1.26785,-2.71892) circle[radius=.9pt];
\fill[black] (1.26785,-2.71892) circle[radius=.9pt];
\fill[black] (2.45746,-1.72073) circle[radius=.9pt];
\node[text=cyclered,align=center] at (-3.46000,0.00000) {$S_1$};
\node[text=cyclered,align=center] at (0.00000,3.72000) {$S_2$};
\node[text=cyclered,align=center] at (3.46000,0.00000) {$S_3$};
\node[text=cyclered] at (0.00000,-3.38000) {$S_4$\ {\scriptsize\color{black}}};
\node[] at (2.31,2.65) {$P_1'$};
\node[] at (-2.37,-2.66) {$P_2'$};
\node[] at (-2.42,2.70) {$Q'$};
\draw[retained] (-2.98357,0.31359) .. controls (-2.63329,0.50324) and (-1.86860,0.78941) .. (-1.89670,0.71987);
\draw[retained] (-0.31359,2.98357) .. controls (-0.74611,2.24330) and (-1.78941,0.86860) .. (-1.71987,0.89670);
\draw[retained] (-1.89670,0.71987) arc[start angle=158.0,end angle=471.99999999999994,radius=0.32];
\fill[black] (-2.98357,0.31359) circle[radius=1.05pt];
\fill[black] (-0.31359,2.98357) circle[radius=1.05pt];
\fill[black] (-1.89670,0.71987) circle[radius=1.05pt];
\fill[black] (-1.71987,0.89670) circle[radius=1.05pt];
\node[text=black] at (-1.60000,0.60000) {$L_{1}$};
\draw[retained] (-2.95442,-0.52094) .. controls (-2.15601,-0.55935) and (-0.55600,-0.82387) .. (-0.62503,-0.79456);
\draw[retained] (2.95442,-0.52094) .. controls (1.82694,-0.66879) and (-0.44400,-0.82387) .. (-0.37497,-0.79456);
\draw[retained] (-0.37497,-0.79456) arc[start angle=-67.0,end angle=247.0,radius=0.32];
\fill[black] (-2.95442,-0.52094) circle[radius=1.05pt];
\fill[black] (2.95442,-0.52094) circle[radius=1.05pt];
\fill[black] (-0.62503,-0.79456) circle[radius=1.05pt];
\fill[black] (-0.37497,-0.79456) circle[radius=1.05pt];
\node[text=black] at (-0.50000,-0.50000) {$L_{2}$};
\draw[retained] (0.31359,2.98357) .. controls (0.51685,2.30747) and (1.87992,0.87505) .. (1.80945,0.90070);
\draw[retained] (2.95442,0.52094) .. controls (2.61563,0.53833) and (1.96183,0.79867) .. (1.99233,0.73016);
\draw[retained] (1.80945,0.90070) arc[start angle=70.00000000000001,end angle=384.00000000000006,radius=0.32];
\fill[black] (0.31359,2.98357) circle[radius=1.05pt];
\fill[black] (2.95442,0.52094) circle[radius=1.05pt];
\fill[black] (1.80945,0.90070) circle[radius=1.05pt];
\fill[black] (1.99233,0.73016) circle[radius=1.05pt];
\node[text=black] at (1.70000,0.60000) {$L_{3}$};
\draw[retained] (-2.69638,1.31511) .. controls (-2.48533,1.52395) and (-1.95107,1.86210) .. (-1.98512,1.79528);
\draw[retained] (-0.82691,2.88379) .. controls (-1.11749,2.52866) and (-1.86528,1.93409) .. (-1.79356,1.95602);
\draw[retained] (-1.98512,1.79528) arc[start angle=153.0,end angle=467.00000000000006,radius=0.32];
\fill[black] (-2.69638,1.31511) circle[radius=1.05pt];
\fill[black] (-0.82691,2.88379) circle[radius=1.05pt];
\fill[black] (-1.98512,1.79528) circle[radius=1.05pt];
\fill[black] (-1.79356,1.95602) circle[radius=1.05pt];
\node[text=black] at (-1.70000,1.65000) {$L_{4}$};
\draw[retained] (-2.69638,-1.31511) .. controls (-1.65587,-1.18895) and (0.44400,-1.07613) .. (0.37497,-1.10544);
\draw[retained] (2.69638,-1.31511) .. controls (1.98658,-1.29854) and (0.55600,-1.07613) .. (0.62503,-1.10544);
\draw[retained] (0.37497,-1.10544) arc[start angle=113.0,end angle=427.0,radius=0.32];
\fill[black] (-2.69638,-1.31511) circle[radius=1.05pt];
\fill[black] (2.69638,-1.31511) circle[radius=1.05pt];
\fill[black] (0.37497,-1.10544) circle[radius=1.05pt];
\fill[black] (0.62503,-1.10544) circle[radius=1.05pt];
\node[text=black] at (0.50000,-1.40000) {$L_{5}$};
\draw[retained] (0.82691,2.88379) .. controls (1.19366,2.60802) and (1.86528,1.93409) .. (1.79356,1.95602);
\draw[retained] (2.69638,1.31511) .. controls (2.42378,1.43278) and (1.95107,1.86210) .. (1.98512,1.79528);
\draw[retained] (1.79356,1.95602) arc[start angle=73.0,end angle=387.0,radius=0.32];
\fill[black] (0.82691,2.88379) circle[radius=1.05pt];
\fill[black] (2.69638,1.31511) circle[radius=1.05pt];
\fill[black] (1.79356,1.95602) circle[radius=1.05pt];
\fill[black] (1.98512,1.79528) circle[radius=1.05pt];
\node[text=black] at (1.70000,1.65000) {$L_{6}$};
\draw[retained] (2.54414,-1.58976) .. controls (2.31902,-1.80067) and (1.75068,-2.13527) .. (1.79423,-2.07421);
\draw[retained] (0.26147,-2.98858) .. controls (0.68655,-2.67206) and (1.66852,-2.18562) .. (1.59434,-2.19670);
\draw[retained] (1.79423,-2.07421) arc[start angle=-35.5,end angle=278.50000000000006,radius=0.3];
\fill[black] (2.54414,-1.58976) circle[radius=1.05pt];
\fill[black] (0.26147,-2.98858) circle[radius=1.05pt];
\fill[black] (1.79423,-2.07421) circle[radius=1.05pt];
\fill[black] (1.59434,-2.19670) circle[radius=1.05pt];
\node[text=black] at (1.55000,-1.90000) {$L_{7}$};
\draw[retained] (-2.54414,-1.58976) .. controls (-2.31902,-1.80067) and (-1.75068,-2.13527) .. (-1.79423,-2.07421);
\draw[retained] (-0.26147,-2.98858) .. controls (-0.68655,-2.67206) and (-1.66852,-2.18562) .. (-1.59434,-2.19670);
\draw[retained] (-1.59434,-2.19670) arc[start angle=261.49999999999994,end angle=575.5,radius=0.3];
\fill[black] (-2.54414,-1.58976) circle[radius=1.05pt];
\fill[black] (-0.26147,-2.98858) circle[radius=1.05pt];
\fill[black] (-1.79423,-2.07421) circle[radius=1.05pt];
\fill[black] (-1.59434,-2.19670) circle[radius=1.05pt];
\node[text=black] at (-1.55000,-1.90000) {$L_{8}$};
\node at (0,-3.88) {(b)};
\end{scope}
\begin{scope}[shift={(0,-8.4)}]
\tikzset{ordinary/.style={draw=black!35,line width=.60pt},retained/.style={draw=black!35,line width=.95pt}}
\Qarc{retained}{35}{65}
\Qarc{ordinary}{115}{145}
\Qarc{retained}{215}{245}
\Qarc{ordinary}{295}{325}
\Qarc{selected}{145}{215}
\Qarc{selected}{65}{115}
\Qarc{selected}{325}{395}
\Qarc{selected}{245}{295}
\fill[black!35] (2.45746,1.72073) circle[radius=.9pt];
\fill[black!35] (1.26785,2.71892) circle[radius=.9pt];
\fill[black!35] (-1.26785,2.71892) circle[radius=.9pt];
\fill[black!35] (-2.45746,1.72073) circle[radius=.9pt];
\fill[black!35] (-2.45746,-1.72073) circle[radius=.9pt];
\fill[black!35] (-1.26785,-2.71892) circle[radius=.9pt];
\fill[black!35] (1.26785,-2.71892) circle[radius=.9pt];
\fill[black!35] (2.45746,-1.72073) circle[radius=.9pt];
\node[text=cyclered,align=center] at (-3.46000,0.00000) {$S_1$\\{\scriptsize\color{black}}};
\node[text=cyclered,align=center] at (0.00000,3.72000) {$S_2$};
\node[text=cyclered,align=center] at (3.46000,0.00000) {$S_3$};
\node[text=cyclered] at (0.00000,-3.38000) {$S_4$\ {\scriptsize\color{black}}};
\node[text=black!40] at (2.31,2.65) {$P_1'$};
\node[text=black!40] at (-2.37,-2.66) {$P_2'$};
\draw[draw=black,line width=1.15pt] (-2.98357,0.31359) .. controls (-2.63329,0.50324) and (-1.86860,0.78941) .. (-1.89670,0.71987);
\draw[draw=black,line width=1.15pt] (-0.31359,2.98357) .. controls (-0.74611,2.24330) and (-1.78941,0.86860) .. (-1.71987,0.89670);
\draw[draw=black,line width=1.15pt] (-1.89670,0.71987) arc[start angle=158.0,end angle=471.99999999999994,radius=0.32];
\fill[white] (-2.98357,0.31359) circle[radius=2.0pt];
\fill[black] (-2.98357,0.31359) circle[radius=1.45pt];
\fill[white] (-0.31359,2.98357) circle[radius=2.0pt];
\fill[black] (-0.31359,2.98357) circle[radius=1.45pt];
\fill[black] (-1.89670,0.71987) circle[radius=1.05pt];
\fill[black] (-1.71987,0.89670) circle[radius=1.05pt];
\node[text=black] at (-1.60000,0.60000) {$L_{1}$};
\draw[draw=black,line width=1.15pt] (-2.95442,-0.52094) .. controls (-2.15601,-0.55935) and (-0.55600,-0.82387) .. (-0.62503,-0.79456);
\draw[draw=black,line width=1.15pt] (2.95442,-0.52094) .. controls (1.82694,-0.66879) and (-0.44400,-0.82387) .. (-0.37497,-0.79456);
\draw[draw=black,line width=1.15pt] (-0.37497,-0.79456) arc[start angle=-67.0,end angle=247.0,radius=0.32];
\fill[white] (-2.95442,-0.52094) circle[radius=2.0pt];
\fill[black] (-2.95442,-0.52094) circle[radius=1.45pt];
\fill[white] (2.95442,-0.52094) circle[radius=2.0pt];
\fill[black] (2.95442,-0.52094) circle[radius=1.45pt];
\fill[black] (-0.62503,-0.79456) circle[radius=1.05pt];
\fill[black] (-0.37497,-0.79456) circle[radius=1.05pt];
\node[text=black] at (-0.50000,-0.50000) {$L_{2}$};
\draw[draw=black,line width=1.15pt] (0.31359,2.98357) .. controls (0.51685,2.30747) and (1.87992,0.87505) .. (1.80945,0.90070);
\draw[draw=black,line width=1.15pt] (2.95442,0.52094) .. controls (2.61563,0.53833) and (1.96183,0.79867) .. (1.99233,0.73016);
\draw[draw=black,line width=1.15pt] (1.80945,0.90070) arc[start angle=70.00000000000001,end angle=384.00000000000006,radius=0.32];
\fill[white] (0.31359,2.98357) circle[radius=2.0pt];
\fill[black] (0.31359,2.98357) circle[radius=1.45pt];
\fill[white] (2.95442,0.52094) circle[radius=2.0pt];
\fill[black] (2.95442,0.52094) circle[radius=1.45pt];
\fill[black] (1.80945,0.90070) circle[radius=1.05pt];
\fill[black] (1.99233,0.73016) circle[radius=1.05pt];
\node[text=black] at (1.70000,0.60000) {$L_{3}$};
\draw[draw=black,line width=1.15pt] (-2.69638,1.31511) .. controls (-2.48533,1.52395) and (-1.95107,1.86210) .. (-1.98512,1.79528);
\draw[draw=black,line width=1.15pt] (-0.82691,2.88379) .. controls (-1.11749,2.52866) and (-1.86528,1.93409) .. (-1.79356,1.95602);
\draw[draw=black,line width=1.15pt] (-1.98512,1.79528) arc[start angle=153.0,end angle=467.00000000000006,radius=0.32];
\fill[white] (-2.69638,1.31511) circle[radius=2.0pt];
\fill[black] (-2.69638,1.31511) circle[radius=1.45pt];
\fill[white] (-0.82691,2.88379) circle[radius=2.0pt];
\fill[black] (-0.82691,2.88379) circle[radius=1.45pt];
\fill[black] (-1.98512,1.79528) circle[radius=1.05pt];
\fill[black] (-1.79356,1.95602) circle[radius=1.05pt];
\node[text=black] at (-1.70000,1.65000) {$L_{4}$};
\draw[draw=black,line width=1.15pt] (-2.69638,-1.31511) .. controls (-1.65587,-1.18895) and (0.44400,-1.07613) .. (0.37497,-1.10544);
\draw[draw=black,line width=1.15pt] (2.69638,-1.31511) .. controls (1.98658,-1.29854) and (0.55600,-1.07613) .. (0.62503,-1.10544);
\draw[draw=black,line width=1.15pt] (0.37497,-1.10544) arc[start angle=113.0,end angle=427.0,radius=0.32];
\fill[white] (-2.69638,-1.31511) circle[radius=2.0pt];
\fill[black] (-2.69638,-1.31511) circle[radius=1.45pt];
\fill[white] (2.69638,-1.31511) circle[radius=2.0pt];
\fill[black] (2.69638,-1.31511) circle[radius=1.45pt];
\fill[black] (0.37497,-1.10544) circle[radius=1.05pt];
\fill[black] (0.62503,-1.10544) circle[radius=1.05pt];
\node[text=black] at (0.50000,-1.40000) {$L_{5}$};
\draw[draw=black,line width=1.15pt] (0.82691,2.88379) .. controls (1.19366,2.60802) and (1.86528,1.93409) .. (1.79356,1.95602);
\draw[draw=black,line width=1.15pt] (2.69638,1.31511) .. controls (2.42378,1.43278) and (1.95107,1.86210) .. (1.98512,1.79528);
\draw[draw=black,line width=1.15pt] (1.79356,1.95602) arc[start angle=73.0,end angle=387.0,radius=0.32];
\fill[white] (0.82691,2.88379) circle[radius=2.0pt];
\fill[black] (0.82691,2.88379) circle[radius=1.45pt];
\fill[white] (2.69638,1.31511) circle[radius=2.0pt];
\fill[black] (2.69638,1.31511) circle[radius=1.45pt];
\fill[black] (1.79356,1.95602) circle[radius=1.05pt];
\fill[black] (1.98512,1.79528) circle[radius=1.05pt];
\node[text=black] at (1.70000,1.65000) {$L_{6}$};
\draw[omitted] (2.54414,-1.58976) .. controls (2.31902,-1.80067) and (1.75068,-2.13527) .. (1.79423,-2.07421);
\draw[omitted] (0.26147,-2.98858) .. controls (0.68655,-2.67206) and (1.66852,-2.18562) .. (1.59434,-2.19670);
\draw[omitted] (1.79423,-2.07421) arc[start angle=-35.5,end angle=278.50000000000006,radius=0.3];
\fill[black!28] (2.54414,-1.58976) circle[radius=1.05pt];
\fill[black!28] (0.26147,-2.98858) circle[radius=1.05pt];
\fill[black!28] (1.79423,-2.07421) circle[radius=1.05pt];
\fill[black!28] (1.59434,-2.19670) circle[radius=1.05pt];
\node[text=black!28] at (1.55000,-1.90000) {$L_{7}$};
\draw[omitted] (-2.54414,-1.58976) .. controls (-2.31902,-1.80067) and (-1.75068,-2.13527) .. (-1.79423,-2.07421);
\draw[omitted] (-0.26147,-2.98858) .. controls (-0.68655,-2.67206) and (-1.66852,-2.18562) .. (-1.59434,-2.19670);
\draw[omitted] (-1.59434,-2.19670) arc[start angle=261.49999999999994,end angle=575.5,radius=0.3];
\fill[black!28] (-2.54414,-1.58976) circle[radius=1.05pt];
\fill[black!28] (-0.26147,-2.98858) circle[radius=1.05pt];
\fill[black!28] (-1.79423,-2.07421) circle[radius=1.05pt];
\fill[black!28] (-1.59434,-2.19670) circle[radius=1.05pt];
\node[text=black!28] at (-1.55000,-1.90000) {$L_{8}$};
\node at (0,-3.88) {(c)};
\end{scope}
\begin{scope}[shift={(8.2,-8.4)}]
\Qarc{draw=black,line width=.95pt}{35}{65}
\Qarc{draw=black,line width=.95pt}{115}{145}
\Qarc{draw=black,line width=.95pt}{215}{245}
\Qarc{draw=black,line width=.95pt}{295}{325}
\Qarc{draw=black,line width=1.25pt}{145}{154}
\Qarc{omitted}{154}{174}
\Qarc{draw=cycleblue,line width=1.25pt}{174}{190}
\Qarc{omitted}{190}{206}
\Qarc{draw=black,line width=1.25pt}{206}{215}
\Qarc{draw=black,line width=1.25pt}{65}{74}
\Qarc{omitted}{74}{84}
\Qarc{draw=cycleblue,line width=1.25pt}{84}{96}
\Qarc{omitted}{96}{106}
\Qarc{draw=black,line width=1.25pt}{106}{115}
\Qarc{draw=black,line width=1.25pt}{325}{334}
\Qarc{omitted}{334}{350}
\Qarc{draw=cycleblue,line width=1.25pt}{350}{370}
\Qarc{omitted}{370}{386}
\Qarc{draw=black,line width=1.25pt}{386}{395}
\Qarc{draw=black,line width=1.25pt}{245}{295}
\fill[black] (2.45746,1.72073) circle[radius=.9pt];
\fill[black] (1.26785,2.71892) circle[radius=.9pt];
\fill[black] (-1.26785,2.71892) circle[radius=.9pt];
\fill[black] (-2.45746,1.72073) circle[radius=.9pt];
\fill[black] (-2.45746,-1.72073) circle[radius=.9pt];
\fill[black] (-1.26785,-2.71892) circle[radius=.9pt];
\fill[black] (1.26785,-2.71892) circle[radius=.9pt];
\fill[black] (2.45746,-1.72073) circle[radius=.9pt];
\node[text=cyclered] at (-3.46000,0.00000) {$S_1$};
\node[text=cyclered] at (0.00000,3.72000) {$S_2$};
\node[text=cyclered] at (3.46000,0.00000) {$S_3$};
\node[text=cyclered] at (0.00000,-3.38000) {$S_4$};
\draw[draw=cycleblue,line width=1.25pt] (-2.98357,0.31359) .. controls (-2.63329,0.50324) and (-1.86860,0.78941) .. (-1.89670,0.71987);
\draw[draw=cycleblue,line width=1.25pt] (-0.31359,2.98357) .. controls (-0.74611,2.24330) and (-1.78941,0.86860) .. (-1.71987,0.89670);
\draw[draw=cycleblue,line width=1.25pt] (-1.89670,0.71987) arc[start angle=158.0,end angle=471.99999999999994,radius=0.32];
\fill[cycleblue] (-2.98357,0.31359) circle[radius=1.05pt];
\fill[cycleblue] (-0.31359,2.98357) circle[radius=1.05pt];
\fill[cycleblue] (-1.89670,0.71987) circle[radius=1.05pt];
\fill[cycleblue] (-1.71987,0.89670) circle[radius=1.05pt];
\node[text=cycleblue] at (-1.60000,0.60000) {$L_{1}$};
\draw[draw=cycleblue,line width=1.25pt] (-2.95442,-0.52094) .. controls (-2.15601,-0.55935) and (-0.55600,-0.82387) .. (-0.62503,-0.79456);
\draw[draw=cycleblue,line width=1.25pt] (2.95442,-0.52094) .. controls (1.82694,-0.66879) and (-0.44400,-0.82387) .. (-0.37497,-0.79456);
\draw[draw=cycleblue,line width=1.25pt] (-0.37497,-0.79456) arc[start angle=-67.0,end angle=247.0,radius=0.32];
\fill[cycleblue] (-2.95442,-0.52094) circle[radius=1.05pt];
\fill[cycleblue] (2.95442,-0.52094) circle[radius=1.05pt];
\fill[cycleblue] (-0.62503,-0.79456) circle[radius=1.05pt];
\fill[cycleblue] (-0.37497,-0.79456) circle[radius=1.05pt];
\node[text=cycleblue] at (-0.50000,-0.50000) {$L_{2}$};
\draw[draw=cycleblue,line width=1.25pt] (0.31359,2.98357) .. controls (0.51685,2.30747) and (1.87992,0.87505) .. (1.80945,0.90070);
\draw[draw=cycleblue,line width=1.25pt] (2.95442,0.52094) .. controls (2.61563,0.53833) and (1.96183,0.79867) .. (1.99233,0.73016);
\draw[draw=cycleblue,line width=1.25pt] (1.80945,0.90070) arc[start angle=70.00000000000001,end angle=384.00000000000006,radius=0.32];
\fill[cycleblue] (0.31359,2.98357) circle[radius=1.05pt];
\fill[cycleblue] (2.95442,0.52094) circle[radius=1.05pt];
\fill[cycleblue] (1.80945,0.90070) circle[radius=1.05pt];
\fill[cycleblue] (1.99233,0.73016) circle[radius=1.05pt];
\node[text=cycleblue] at (1.70000,0.60000) {$L_{3}$};
\draw[draw=black,line width=1.25pt] (-2.69638,1.31511) .. controls (-2.48533,1.52395) and (-1.95107,1.86210) .. (-1.98512,1.79528);
\draw[draw=black,line width=1.25pt] (-0.82691,2.88379) .. controls (-1.11749,2.52866) and (-1.86528,1.93409) .. (-1.79356,1.95602);
\draw[draw=black,line width=1.25pt] (-1.98512,1.79528) arc[start angle=153.0,end angle=467.00000000000006,radius=0.32];
\fill[black] (-2.69638,1.31511) circle[radius=1.05pt];
\fill[black] (-0.82691,2.88379) circle[radius=1.05pt];
\fill[black] (-1.98512,1.79528) circle[radius=1.05pt];
\fill[black] (-1.79356,1.95602) circle[radius=1.05pt];
\node[text=black] at (-1.70000,1.65000) {$L_{4}$};
\draw[draw=black,line width=1.25pt] (-2.69638,-1.31511) .. controls (-1.65587,-1.18895) and (0.44400,-1.07613) .. (0.37497,-1.10544);
\draw[draw=black,line width=1.25pt] (2.69638,-1.31511) .. controls (1.98658,-1.29854) and (0.55600,-1.07613) .. (0.62503,-1.10544);
\draw[draw=black,line width=1.25pt] (0.37497,-1.10544) arc[start angle=113.0,end angle=427.0,radius=0.32];
\fill[black] (-2.69638,-1.31511) circle[radius=1.05pt];
\fill[black] (2.69638,-1.31511) circle[radius=1.05pt];
\fill[black] (0.37497,-1.10544) circle[radius=1.05pt];
\fill[black] (0.62503,-1.10544) circle[radius=1.05pt];
\node[text=black] at (0.50000,-1.40000) {$L_{5}$};
\draw[draw=black,line width=1.25pt] (0.82691,2.88379) .. controls (1.19366,2.60802) and (1.86528,1.93409) .. (1.79356,1.95602);
\draw[draw=black,line width=1.25pt] (2.69638,1.31511) .. controls (2.42378,1.43278) and (1.95107,1.86210) .. (1.98512,1.79528);
\draw[draw=black,line width=1.25pt] (1.79356,1.95602) arc[start angle=73.0,end angle=387.0,radius=0.32];
\fill[black] (0.82691,2.88379) circle[radius=1.05pt];
\fill[black] (2.69638,1.31511) circle[radius=1.05pt];
\fill[black] (1.79356,1.95602) circle[radius=1.05pt];
\fill[black] (1.98512,1.79528) circle[radius=1.05pt];
\node[text=black] at (1.70000,1.65000) {$L_{6}$};
\node[font=\scriptsize] at (0.93165,3.24906) {$v_1$};
\node[font=\scriptsize] at (0.35331,3.36148) {$v_2$};
\node[font=\scriptsize] at (-0.35331,3.36148) {$v_3$};
\node[font=\scriptsize] at (-0.93165,3.24906) {$v_4$};
\node at (0,-3.88) {(d)};
\end{scope}
\draw[-{Latex[length=1.8mm,width=1.3mm]},line width=.65pt] (3.80,.08)--(4.40,.08);
\draw[-{Latex[length=1.8mm,width=1.3mm]},line width=.65pt] (3.80,-8.32)--(4.40,-8.32);
\draw[-{Latex[length=1.8mm,width=1.3mm]},line width=.65pt]
 (4.90,-3.372)--(3.30,-4.908);
\end{tikzpicture}
\endgroup
\caption{Constructing disjoint cycles using a star with the leaves $D_i$ and the center $Q$.}
\label{fig:cycles-to-paths}
\end{figure}

We next construct a family of disjoint cycles in the union
of $Q'$ and the paths $L_i$, keeping every edge of $Q'$ that
lies on no member of $\mathcal S$ and all but at most
$|\mathcal S|$ of the paths $L_i$.
The paths $L_i$ are disjoint and meet $Q'$ only at their
endpoints on members of $\mathcal S$: besides $Q$, the
construction of $Q'$ used only cycles $D_j\in\mathcal D$ and paths in
their families $\mathcal P_j$, whereas every $L_i$ uses a cycle
$D_i\notin\mathcal D$.
Let $F_Q$ be an auxiliary multigraph with vertex set $\mathcal S$.
Each path $L_i$ gives an edge joining the two distinct
members of $\mathcal S$ containing its endpoints.
Choose an inclusion-maximal spanning subgraph $F'_Q$ of $F_Q$
in which every degree is even.
As for $F'$, maximality implies that the edges outside
$F'_Q$ contain no cycle.
Let $\mathcal L_Q$ be the family of paths $L_i$ corresponding
to the edges of $F'_Q$.
Then $\mathcal L_Q$ contains all but at most $|\mathcal S|$
of the paths $L_i$, and each member of $\mathcal S$ contains
an even number of endpoints of paths in $\mathcal L_Q$;
see Figure~\ref{fig:cycles-to-paths}(c).

For each $P\in\mathcal S$, list the endpoints on $P$
of paths in $\mathcal L_Q$ as $v_1,\ldots,v_{2r}$
in order along $P$.
For $1\le j\le r$, delete from $Q'$ the edges and internal
vertices of the subpath of $P$ from $v_{2j-1}$ to $v_{2j}$;
when $r=0$, leave $P$ unchanged.
After adding the paths in $\mathcal L_Q$, each listed endpoint
has one incident edge from $Q'$ and one from $\mathcal L_Q$,
and every other remaining vertex also has degree two.
This completes the required family of disjoint cycles, containing
all paths in $\mathcal L_Q$ and every edge of $Q'$ outside
the members of $\mathcal S$; see Figure~\ref{fig:cycles-to-paths}(d).

We remove the cycles of length less than $3x/2$ from the
family above and bound the number of vertices removed.
Suppose first that $\mathcal S\ne\emptyset$.
We show that at most $|\mathcal S|$ cycles are removed.
A constructed cycle consisting only of subpaths of members
of $\mathcal S$ and paths in $\mathcal L_Q$ contains at least
two paths in $\mathcal L_Q$, since each path in $\mathcal L_Q$
joins distinct members of $\mathcal S$.
Its length is therefore at least $2(1-2/s)x\ge3x/2$.
Every removed cycle must consequently contain an edge of
$Q'$ lying on no member of $\mathcal S$.
These edges form $|\mathcal S|$ disjoint paths, each contained
in a single constructed cycle.
Hence at most $|\mathcal S|$ cycles are removed, containing
at most $3x|\mathcal S|/2$ vertices in total.

If $\mathcal S=\emptyset$, every $D_i\notin\mathcal D$
belongs to $\mathcal E$, and the only resulting cycle is $Q'$.
By Claim~\ref{clm:insert-path} and~\eqref{eq:star-selection},
\[
 |Q'|\ge(1-2/s)\left(M_Q-\sum_{D\in\mathcal E}|D|\right)
 \ge(1-7/s)M_Q\ge3x/2,
\]
where the last inequality uses $M_Q\ge2x$ and $s\ge64$.
Hence no cycle is removed in this case.
Thus the number of vertices removed is at most
$3x|\mathcal S|/2\le2x|\mathcal S|$ in both cases.
We now bound the number of vertices of
$Q\cup D_1\cup\cdots\cup D_t$ not covered by the remaining cycles.
The number of uncovered vertices on members
of $\mathcal S$ and on cycles in $\mathcal E$ is at most
$\sum_{P\in\mathcal S}|P|+\sum_{D\in\mathcal E}|D|$.
At most $|\mathcal S|$ paths $L_i$ lie outside $\mathcal L_Q$,
and the corresponding cycles $D_i$ contain at most
$2x|\mathcal S|$ vertices in total.
As shown above, the deletion of cycles of length less than
$3x/2$ removed at most $2x|\mathcal S|$ vertices.
The sum of these bounds is at most $5M_Q/s$
by~\eqref{eq:star-selection}.

For each $P\in\mathcal Q\setminus\mathcal S$,
Claim~\ref{clm:insert-path} bounds the number of vertices
omitted from $P$ and the cycles in $\mathcal D_P$
by a $2/s$ fraction of their total.
For each $L_i\in\mathcal L_Q$,
Claim~\ref{clm:between-parts} gives the same bound
for the vertices omitted from $D_i$.
The original vertex sets to which these bounds apply
are pairwise disjoint, so these additional omissions
total at most $2M_Q/s$.
This completes the construction for the star centered at $Q$:
the resulting cycles are disjoint, each has length at least
$3x/2$, and together they cover at least $(1-7/s)M_Q$
vertices of $Q\cup D_1\cup\cdots\cup D_t$.

\medskip

The cycles in the final family come from three sources:
\begin{itemize}[leftmargin=*]
\item the cycles constructed from $\mathcal B\cup\mathcal R$;
\item the remaining cycles from the constructions in the stars;
\item the original long cycles outside $\mathcal B$ that
belong to no star.
\end{itemize}
Every cycle in this family has length at least $3x/2$.
The cycles are pairwise disjoint:
the constructions use different members of $\mathcal H$,
and the paths from $\mathcal P$ used in the constructions
are disjoint and have all internal vertices outside every
member of $\mathcal H$.

The original long cycles in the last item are covered entirely.
Summing the bounds for the construction from
$\mathcal B\cup\mathcal R$ and for the individual stars,
the number of original vertices not covered by the final
family is at most
\[
 \frac{6M}{s}
 +\frac2s\sum_{D\in\mathcal R\setminus\mathcal B}|D|
 +\frac7s\sum_Q M_Q
 \le \frac{13M}{s},
\]
where $Q$ ranges over the centers of the chosen stars.
For the last inequality, the cycles in
$\mathcal R\setminus\mathcal B$ and those belonging to
the stars are pairwise disjoint, so
$\sum_{D\in\mathcal R\setminus\mathcal B}|D|
+\sum_Q M_Q\le M$.
Thus the final family has total length at least
\[
 (1-13/s)M\ge(1-13x^{-1/80})M.
\]
The family is nonempty because $s\ge64$, completing the proof.
\end{proof}

\subsection{From cycle families to a long cycle}\label{sec:criterion-completion}
Iterating Lemma~\ref{lem:prefix-bounded} gives the following
bound for a single cycle.

\begin{proposition}\label{prop:merge-family}
There is an integer $L_0$, depending only on $K$, such that
for every integer $L\ge L_0$, if $G$ contains a nonempty family
of disjoint cycles of minimum length at least $L$ and total
length $M$, then $c(G)\ge(1-L^{-1/90})M$.
\end{proposition}
\begin{proof}
Choose $L_0$ sufficiently large, above the threshold in
Lemma~\ref{lem:prefix-bounded}, with $13L_0^{-1/80}<1$.
Starting with the given family, repeatedly apply that lemma
with $x$ equal to the current minimum cycle length while at
least two cycles remain.
After $j$ applications the minimum length is at least
$(3/2)^jL$, so the process terminates with one cycle, since
no cycle has more than $n$ vertices.
If there are $r$ applications, the final cycle has length at least
\[
\begin{aligned}
 M\prod_{j=0}^{r-1}
   \left(1-13L^{-1/80}(2/3)^{j/80}\right)
 &\ge M\left(1-13L^{-1/80}
       \sum_{j=0}^{\infty}(2/3)^{j/80}\right)\ge (1-L^{-1/90})M.
\end{aligned}
\]
The first inequality uses $\prod_j(1-a_j)\ge1-\sum_j a_j$
for $0\le a_j\le1$.
The last holds for sufficiently large $L_0$, since the geometric
series converges and $1/80>1/90$.
\end{proof}

We are ready to prove Theorem~\ref{thm:general}.

\begin{proof}[\normalfont\bfseries Proof of Theorem~\ref{thm:general}]
Choose an integer $L\ge L_0$ such that $L^{-1/90}\le1/2$,
where $L_0$ is given by Proposition~\ref{prop:merge-family}.
Lemma~\ref{lem:prefix-initial} supplies, for all sufficiently
large $n$, a family of disjoint cycles of minimum length at
least $L$ and total length at least $\nu n$, where $\nu>0$
depends only on $K$.
Proposition~\ref{prop:merge-family} then gives $c(G)\ge\nu n/2$.

Now suppose that $G$ is $d$-regular, with $d$ sufficiently large
in terms of $K$, and let $L=\lfloor d^{19/20}\rfloor$.
Replace each edge of $G$ by its two orientations and apply
Theorem~\ref{thm:few-cycles}. Some spanning family has
$O(n\log d/d)$ directed cycles. Each directed cycle of length
two gives a single edge of $G$, and every longer one gives an
undirected cycle. The cycles of length at least $L$ therefore
have total length at least $n-O(Ln\log d/d)$, since every other
member has fewer than $L$ vertices.
For sufficiently large $d$, this family is nonempty and $L\ge L_0$.
Proposition~\ref{prop:merge-family} gives
\[
 c(G)\ge\bigl(1-L^{-1/90}-O(L\log d/d)\bigr)n
 \ge(1-d^{-1/100})n,
\]
where the last inequality follows from
\[
 L^{-1/90}+O(L\log d/d)
 =O\bigl(d^{-19/1800}+d^{-1/20}\log d\bigr)
 =o(d^{-1/100}).\qedhere
\]
\end{proof}

\section{Vertex-transitive graphs}\label{sec:vertex-transitive}
In this section we prove Theorem~\ref{thm:main}.
Recall that $\beta_r=|B_G(v,r)|$ is independent of the
center $v$ when $G$ is vertex-transitive.
We write $\Aut(G)$ for the group of automorphisms of $G$,
that is, the permutations of $V(G)$ that preserve adjacency.

We first use $\beta_r$ to bound the number of components of
$G-S$ with at least a prescribed number of vertices.
If $\beta_R\le R^{5/4}$ for some sufficiently large
$R<n^{4/5}$, where $n=|G|$, we use a theorem of DeVos and
Mohar to find a cycle missing at most $n^{19/20}$ vertices.
Otherwise, the bounds on the number of components allow us
to verify~\eqref{eq:component-condition} and apply
Theorem~\ref{thm:general}.

\subsection{A general bound on the number of components}\label{app:ball}

We first establish a bound that holds for every finite
connected vertex-transitive graph.
\begin{lemma}\label{lem:ballcap}
Let $G$ be a finite connected vertex-transitive graph and let
$r\ge1$ be an integer. For every $S\subseteq V(G)$ and integer
$1\le\lambda\le\beta_r$, the number of components of $G-S$
other than a largest one that have at least $\lambda$ vertices
is at most $(4r+3)|S|/\lambda$.
\end{lemma}

\begin{proof}
Let $N$ be the number of components in the statement.
We may assume that $N>0$.
We first construct an auxiliary tree $T$ with
$V(T)\subseteq V(G)$ and $d_G(u,v)=2r+1$ for every
$uv\in E(T)$.
Initially, $T$ consists of an arbitrary vertex of $G$.
Whenever the balls $B_G(c,2r)$ with $c\in V(T)$ do not cover
$V(G)$, a shortest path from $V(T)$ to an
uncovered vertex contains a vertex $v$ whose minimum distance
to $V(T)$ is exactly $2r+1$.
We extend $T$ by adding $v$ and an edge $uv$, where
$u\in V(T)$ satisfies $d_G(u,v)=2r+1$.
The process ends when these balls cover $V(G)$.
Distinct vertices of $T$ have distance at least $2r+1$ in $G$,
so their balls of radius $r$ are disjoint and
$|V(T)|\le |G|/\beta_r$.
For each edge $e=uv$ of $T$, let $Q_e$ be a shortest path
from $u$ to $v$ in $G$, of length $2r+1$.
The covering by balls of radius $2r$ also allows us to fix,
for every $v\in V(G)$, a path $P_v$ of length at most $2r$
from $v$ to a vertex of $T$.

Fix $\phi\in\Aut(G)$, and write $\phi(P)$ for the path
obtained by applying $\phi$ to every vertex of a path $P$.
Let $p$ be the number of components of $G-S$ with at least
$\lambda$ vertices that contain no vertex of $\phi(V(T))$,
and let $q$ be the number of edges $e\in E(T)$ for which
$\phi(Q_e)$ meets $S$.
We show that 
\begin{equation}\label{eq:ballcap-components}
 N\le p+q.
\end{equation}

To see this, deleting these $q$ edges from $T$ leaves a forest with
$q+1$ components. There is an injection from the components
of $G-S$ meeting $\phi(V(T))$ into the components of this
forest. Indeed, choose one vertex $\phi(c)$ in the intersection
of each such component with $\phi(V(T))$, and map that
component to the component of the forest containing $c$.
If two chosen centers $c,d$ lie in the same component of the
forest, the paths $\phi(Q_e)$ along the $(c,d)$-path in $T$
join $\phi(c)$ to $\phi(d)$ in $G-S$.
Thus $\phi(c)$ and $\phi(d)$ belong to the same component
of $G-S$, which proves injectivity.
Since $N>0$, exactly $N+1$ components of $G-S$ have at least
$\lambda$ vertices, and $N+1-p$ of them meet $\phi(V(T))$.
Hence $N+1-p\le q+1$, which proves~\eqref{eq:ballcap-components}.

For every vertex $\phi(v)$ in one of the $p$ components
containing no center, the path $\phi(P_v)$ ends outside
that component and therefore meets $S$.
These components have at least $\lambda p$ vertices in total, so
\begin{equation}\label{eq:ballcap-paths}
 \lambda p\le
 \bigl|\{v\in V(G):\phi(P_v)\text{ meets }S\}\bigr|.
\end{equation}

We now choose $\phi$ uniformly at random from $\Aut(G)$.
For a fixed vertex $u$ and any $a,b\in V(G)$,
vertex-transitivity gives an automorphism $\psi$ with
$\psi(a)=b$.
The map $\phi\mapsto\psi\phi$ is a bijection from the
automorphisms taking $u$ to $a$ to those taking $u$ to $b$.
Thus $\phi(u)$ is uniformly distributed on $V(G)$.
For any fixed path $P$, the union bound now gives
\[
 \Pr\bigl(\phi(P)\text{ meets }S\bigr)
 \le\sum_{u\in V(P)}\Pr\bigl(\phi(u)\in S\bigr)
 =\frac{|V(P)||S|}{|G|}.
\]
Each $P_v$ has at most $2r+1$ vertices, and each $Q_e$
has $2r+2$ vertices. Moreover, $|E(T)|\le |G|/\beta_r$.
Taking expectations in~\eqref{eq:ballcap-components}
and~\eqref{eq:ballcap-paths}, and using the definition of $q$, gives
\[
\begin{aligned}
 \lambda N
 &\le \lambda\mathbb E p+\lambda\mathbb E q\le (2r+1)|S|
      +\frac{\lambda}{\beta_r}(2r+2)|S|\le (4r+3)|S|,
\end{aligned}
\]
where the last inequality uses $\lambda\le\beta_r$.
\end{proof}

\subsection{Structure from a small ball}\label{sec:smallballs}
If $\beta_r>r^{5/4}$ for every integer $R_0\le r<n^{4/5}$,
where $R_0$ is a sufficiently large absolute constant,
Lemma~\ref{lem:ballcap} allows us to verify
\eqref{eq:component-condition} and apply
Theorem~\ref{thm:general}, as we show in the proof of
Theorem~\ref{thm:main}.
This subsection treats the remaining case, where
$\beta_R\le R^{5/4}$ for some integer $R_0\le R<n^{4/5}$,
using a structural theorem of DeVos and Mohar.
\begin{theorem}\label{prop:early}
There is an absolute integer $R_0$ such that every connected
vertex-transitive graph $G$ on $n$ vertices satisfying
$\beta_R\le R^{5/4}$ for some integer $R_0\le R<n^{4/5}$
satisfies $c(G)\ge n-n^{19/20}$.
\end{theorem}

The proof uses the following three results.
From a cyclically ordered partition supplied by
Theorem~\ref{thm:DM}, we construct a quotient graph of $G$
that is a Cayley graph of a cyclic or a dihedral group.
In the cyclic case, the quotient has a Hamilton cycle;
in the dihedral case, Lemma~\ref{lem:dihedral-cycle}
gives a cycle containing almost all vertices of the quotient graph.
Lemma~\ref{lem:normal-quantitative} then lifts this cycle
to $G$ and gives the required bound on its length.

We give some group-theoretical notation here. For a subgroup $H$ of $\Aut(G)$, the \emph{$H$-orbit} of
a vertex $v$ is $H(v)=\{h(v):h\in H\}$.
These orbits partition $V(G)$.
The quotient graph $G/H$ has the $H$-orbits as its vertices;
two distinct orbits are adjacent if an edge of $G$ joins them.
We write $H\lhd\Aut(G)$ when $H$ is a \emph{normal subgroup}, that is, $\phi H\phi^{-1}=H$ for every $\phi\in\Aut(G)$.
In this case, $\phi(H(v))=H(\phi(v))$ for every $v$ and $\phi$.
Thus, when $G$ is vertex-transitive, all $H$-orbits have the
same size.
The diameter of a nonempty set $W\subseteq V(G)$ in $G$ is
$\max\{d_G(u,v):u,v\in W\}$.
We write $\operatorname{diam}(G)$ for the diameter of $V(G)$ in $G$.

\begin{lemma}\label{lem:normal-quantitative}
Let $G$ be a finite connected vertex-transitive graph, let
$H\lhd\Aut(G)$, and let $\ell\ge0$ be an integer.
Suppose that the $H$-orbits have size $m$ and diameter at most
$\ell$ in $G$. Write $q=c(G/H)$. Then
\[
 c(G)\ge qm-8m\sqrt{qm\beta_{2\ell}}.
\]
\end{lemma}

\begin{proof}
We may assume that the asserted lower bound is positive.
Fix a longest cycle $C_0$ in $G/H$, and let $S$ be the
union of the $H$-orbits corresponding to its vertices.
We first construct at most $m$ disjoint cycles covering $S$.
For each edge of $C_0$, consider the bipartite graph
with the two corresponding $H$-orbits as its parts and
all edges of $G$ between them.
By the definition of the $H$-orbits, vertices in the same
part have equal degrees.
Since both parts have size $m$ and there is at least one
edge between them, this bipartite graph is regular of
positive degree.
Hall's theorem~\cite{Hall} therefore gives a perfect matching.
Choose one such matching for each edge of $C_0$.
Their union has degree two at every vertex of $S$,
so it consists of disjoint cycles covering $S$.
Each vertex has one matching edge for each of the two
edges of $C_0$ incident with its orbit.
Thus, along any of these cycles, the orbits are visited
successively in the cyclic order given by $C_0$.
Every cycle therefore meets all $q$ orbits and has length
at least $q$. Since $|S|=qm$, there are at most $m$ cycles.

We prove the required bound on $c(G)$ by constructing
a cycle $C$ in $G$ as follows.
Initially, let $C$ be one of the cycles covering $S$, and let
$U$ be the set of vertices on the remaining cycles.
At each step, we join $C$ to one remaining cycle in this covering and remove
the vertices of that cycle from $U$, maintaining
$V(C)\cap U=\emptyset$ and $|V(C)\cap S|\ge q$.
The following claim bounds the number of vertices of $S$
lost in each step.

\medskip

\noindent {\bf Claim.}
Let $\mathcal P$ be a nonempty family of disjoint $(C,U)$-paths
whose internal vertices lie outside $V(C)\cup U$, and let
\[
 B=\left\lceil3\sqrt{\frac{|V(C)\cap S|\,|U|}{|\mathcal P|}}
   \right\rceil.
\]
If $2B<q$, then some remaining cycle $D$ can be joined to
$C$ to form a cycle containing at least
$|V(C)\cap S|+|D|-2B$ vertices of $S$ and disjoint from
every remaining cycle other than $D$.

\medskip

\begin{proof}[\normalfont\bfseries Proof of Claim.]
Partition $V(C)$ into $\lceil |V(C)\cap S|/B\rceil$
sets of consecutive vertices, each containing at most
$B$ vertices of $S$.
Since $|V(C)\cap S|\ge q>B$, there are at most
$2|V(C)\cap S|/B$ parts.
Similarly, each remaining cycle $D$ can be partitioned into
$\lceil |D|/B\rceil\le2|D|/B$ sets of consecutive vertices,
each with at most $B$ vertices.
The partitions of all remaining cycles therefore have
at most $2|U|/B$ parts in total.

Each path in $\mathcal P$ determines an ordered pair $(I,J)$,
where $I$ is the part of the partition of $V(C)$ containing
its endpoint on $C$, and $J$ is the part of the partition
of a remaining cycle containing its other endpoint.
The number of possible pairs is at most
\[
 \frac{2|V(C)\cap S|}{B}\cdot\frac{2|U|}{B}
 \le\frac49|\mathcal P|<|\mathcal P|.
\]
Thus two distinct paths $P_1,P_2\in\mathcal P$ determine
the same pair $(I,J)$.
Let $D$ be the remaining cycle containing $J$.
The endpoints of $P_1$ and $P_2$ on $C$ bound a subpath
within $I$ containing at most $B$ vertices of $S$.
Their endpoints on $D$ bound a subpath within $J$
containing at most $B$ vertices.
The complementary subpaths of $C$ and $D$, together with
$P_1$ and $P_2$, form the required cycle.
\end{proof}

Suppose that $U\ne\emptyset$.
We now prove that there are at least $|V(C)\cap S|/\beta_{2\ell}$
disjoint $(C,U)$-paths in $G$.
Choose an inclusion-maximal subset of $V(C)\cap S$
whose distinct vertices have distance greater than
$2\ell$ in $G$.
The balls of radius $2\ell$ centered at the chosen vertices
cover $V(C)\cap S$, so there are at least
$|V(C)\cap S|/\beta_{2\ell}$ chosen vertices.
Every remaining cycle meets all $q$ orbits in $S$.
Thus, for each chosen vertex $v$, the set $U$ contains
a vertex in the same orbit as $v$, and a shortest path from
$v$ to $U$ has length at most $\ell$.
These paths lie in the pairwise disjoint balls $B_G(v,\ell)$.
Let $\mathcal P$ be obtained as follows.
For each chosen vertex $v$, consider the shortest $(v,U)$-path
chosen for $v$, with its vertices ordered from $v$ to its
endpoint in $U$.
Include in $\mathcal P$ the subpath from the last vertex
of $C$ in this order to that endpoint in $U$.
Then the paths in $\mathcal P$ are pairwise disjoint and have all internal vertices
outside $V(C)\cup U$. Then $|\mathcal P|\ge |V(C)\cap S|/\beta_{2\ell}$.

Since $|U|\le qm$, the integer $B$ in the claim satisfies
\[
 B\le\left\lceil3\sqrt{qm\beta_{2\ell}}\right\rceil,
 \qquad
 2B\le8\sqrt{qm\beta_{2\ell}}<q,
\]
where the last inequality follows from the positivity of
the asserted lower bound.
The claim therefore gives a remaining cycle $D$ and a new
cycle containing at least
\[
 |V(C)\cap S|+|D|-8\sqrt{qm\beta_{2\ell}}
 >|V(C)\cap S|
\]
vertices of $S$, since $|D|\ge q$.
Replacing $C$ by this cycle and removing $V(D)$ from $U$
preserves the two properties required of $C$.
We repeat until $U$ is empty.
The cycles constructed from the matchings cover all $qm$
vertices of $S$, and there are at most $m$ of them.
Since each replacement loses at most
$8\sqrt{qm\beta_{2\ell}}$ vertices of $S$, the final cycle
contains at least $qm-8m\sqrt{qm\beta_{2\ell}}$ vertices,
as required.
\end{proof}

For a finite group $\Gamma$ and a symmetric generating set
$\mathcal A\subseteq\Gamma\setminus\{1\}$, the Cayley graph
has vertex set $\Gamma$ and edges $\{g,ga\}$ for
$g\in\Gamma$ and $a\in\mathcal A$.
We write
$D_{2N}=\langle r,f:r^N=f^2=1,\;frf=r^{-1}\rangle$
for the dihedral group of order $2N$.

\begin{lemma}\label{lem:dihedral-cycle}
Let $t,N$ be positive integers with $N>8t^2$, and let $Y$ be a
connected Cayley graph of $D_{2N}$ with generating set $\mathcal A$.
Suppose that either
\begin{enumerate}[label=\textup{(\roman*)}]
\item $r^a\in\mathcal A$ for some integer $a$ with $0<|a|\le t$; or
\item every member of $\mathcal A$ is a reflection, and
$r^bf,r^cf\in\mathcal A$ for some integers $b,c$ with
$0<|b-c|\le2t$.
\end{enumerate}
Then $c(Y)\ge2N-4t$.
\end{lemma}

\begin{proof}
In each case, we first partition $V(Y)$ into cycles
$D_1,\ldots,D_k$ for some $k\le2t$, and then join them
successively, omitting at most two vertices whenever
we include another cycle.
The cycle $C_j$ constructed from $D_1,\ldots,D_j$
will also satisfy $|E(D_i)\setminus E(C_j)|\le2(j-1)$ for every $i\le j$.
We count the available pairs of edges or subpaths as follows to ensure that the construction will continue.

In case~\textup{(i)}, let $g=\gcd(N,a)$.
The edges given by $r^a$ form $k=2g\le2t$ disjoint cycles
of length $N/g$ covering $V(Y)$.
Consider two of these cycles $C,D$ with an edge between
them, and fix a generator $h\in\mathcal A$ joining a
vertex of $C$ to a vertex of $D$.
Right multiplication by $h$ maps $C$ onto $D$ and
preserves their edges, since
$h^{-1}r^ah=r^a$ or $r^{-a}$.
For each edge $uv$ of $C$, replacing $uv$ and
$\{uh,vh\}$ by $\{u,uh\}$ and $\{v,vh\}$ produces
one cycle on $V(C)\cup V(D)$.
The $N/g$ edges of $C$ give $N/g$ such pairs, and each
edge of $C$ or $D$ occurs in exactly one pair.

In case~\textup{(ii)}, after replacing $f$ by one of the
two specified reflections, we may assume that
$f,r^af\in\mathcal A$ with $0<|a|\le2t$.
Write $x_i=r^i$ and $y_i=r^if$, with indices taken modulo
$N$, and let $g=\gcd(N,a)$.
For each residue class modulo $g$, the edges given by
$f$ and $r^af$ form a cycle on the vertices $x_i,y_i$
whose indices belong to that class.
These $k=g\le2t$ cycles partition $V(Y)$, and each has
length $2N/g$.

Suppose that an edge $x_i y_{i+b}$, given by a generator
$r^bf\in\mathcal A$, joins distinct initial cycles
$C,D$, with $x_i\in V(C)$.
The paths $x_i y_{i+a}x_{i+a}$ and
$y_{i+b}x_{i+b}y_{i+a+b}$ lie in $C$ and $D$,
respectively.
Their complementary subpaths, together with the edges
$x_i y_{i+b}$ and $x_{i+a}y_{i+a+b}$, form a single
cycle on $(V(C)\cup V(D))\setminus\{y_{i+a},x_{i+b}\}$.
As $i$ ranges over its residue class modulo $g$,
this gives $N/g$ pairs of subpaths of length two
in the same two cycles.
On either cycle, these subpaths are pairwise
edge-disjoint.

\medskip

We now construct $C_1,\ldots,C_k$.
By connectedness, we may rearrange the initial cycles so
that every $D_{j+1}$ has an edge to some $D_i$ with
$i\le j$.
Let $C_1=D_1$, and suppose that $C_j$ has been constructed
with $j<k$.
For a fixed generator joining $D_{j+1}$ to such a $D_i$,
the preceding construction gives $N/g$ pairs of edges
or subpaths, with no edge of $D_i$ belonging to more
than one pair.
Since at most $2(j-1)$ original edges of $D_i$ are absent
from $C_j$, at least $N/g-2(j-1)>0$ pairs have their
edge or subpath on $D_i$ contained in $C_j$.
Here we use $N/g\ge N/(2t)>4t$ and $j<k\le2t$.

The corresponding edge or subpath lies on $D_{j+1}$,
which is disjoint from $C_j$ and has not yet been used.
The same replacement therefore gives a cycle $C_{j+1}$
using only vertices of $D_1,\ldots,D_{j+1}$.
It deletes at most two original edges from each of
$D_i$ and $D_{j+1}$, and none from the other initial
cycles, so the required bound holds with $j+1$
in place of $j$.

After $k-1$ steps, we obtain $C_k$.
In case~\textup{(i)}, no vertices are omitted, so $C_k$
is a Hamilton cycle.
In case~\textup{(ii)}, each step omits two vertices,
giving $|C_k|\ge2N-2(k-1)\ge2N-4t$.
\end{proof}

The following result of DeVos and Mohar's
theorem supplies the partition of the vertex set
used to construct the quotient graph. The parts of the partition are
called \textit{blocks}.
For blocks $B_0,\ldots,B_{m-1}$ in cyclic order, the cyclic
distance between $B_i$ and $B_j$ is
$\min\{|i-j|,m-|i-j|\}$.
For $A\subseteq V(G)$, write
$\partial A=\{v\notin A:v\text{ has a neighbor in }A\}$,
and let $d_G(v,W)=\min_{w\in W}d_G(v,w)$ for nonempty $W$.

\begin{theorem}[DeVos--Mohar {\cite[Lemma~5.3]{DeVosMohar}}]\label{thm:DM}
Let $G$ be a finite connected vertex-transitive graph, let
$A\subseteq V(G)$, and set $k=|\partial A|$.
Suppose that $\operatorname{diam}(G)\ge31(k+1)^2$, and that
there are vertices $u\in A$ and $v\notin A\cup\partial A$ such that
\[
 d_G(u,V(G)\setminus A)\ge k+1,
 \qquad d_G(v,A\cup\partial A)\ge k+1.
\]
Then there are integers $s,t$ with $st\le k/2$ and a
partition of $V(G)$ into blocks of size $s$, with a cyclic order
in which every automorphism of $G$ permutes the blocks by a rotation
or a reflection, such that
\begin{enumerate}[label=\textup{(\roman*)}]
\item the endpoints of every edge belong to blocks at distance
at most $t$ in the cyclic order;
\item any two vertices in the same block have distance at most
$2st$ in $G$.
\end{enumerate}
\end{theorem}

We proceed to prove Theorem~\ref{prop:early}.
\begin{proof}[\normalfont\bfseries Proof of Theorem~\ref{prop:early}]
Choose an absolute integer $R_0$ sufficiently large, and let
$G,n,R$ satisfy the hypotheses. Fix a vertex $y$ of $G$.
Since $\beta_R\le R^{5/4}<n$, there is a vertex $z$ with
$d_G(y,z)\ge R+1$, and in particular
$\operatorname{diam}(G)>R$.

There are at least $R/5$ integers $r$ with
$R/4\le r\le R/2$, and the sum of
$\beta_{r+1}-\beta_r$ over these integers is at most
$\beta_R\le R^{5/4}$.
We may therefore choose such an $r$ so that
\[
 k=|\partial A|=\beta_{r+1}-\beta_r\le5R^{1/4}.
\]
We now apply Theorem~\ref{thm:DM} to the set $A:=B_G(y,r)$ with the vertices $u=y$ and $v=z$.
To see the hypotheses hold, we have $A\cup\partial A=B_G(y,r+1)$ and $d_G(y,z)\ge R+1$, so
\[
 d_G(y,V(G)\setminus A)\ge r+1\ge R/4,
 \qquad
 d_G(z,A\cup\partial A)\ge R-r\ge R/2.
\]
For sufficiently large $R_0$, we have
$R/4\ge5R^{1/4}+1\ge k+1$ and
$31(k+1)^2\le R<\operatorname{diam}(G)$.
Thus all the hypotheses of Theorem~\ref{thm:DM}
hold, and it then gives a partition into cyclically ordered blocks
of size $s$ that satisfies the two conclusions there with some $t\leq k/2s$.

Let $H$ consist of the automorphisms that fix every block.
Since every automorphism permutes the blocks, we have $H\lhd\Aut(G)$.
The $H$-orbits have a common size $m\le s$ and lie within
blocks, so their diameter in $G$ is at most $2st\le k$.
The quotient group $\Gamma=\Aut(G)/H$ acts transitively on
$G/H$ and can be identified, through its action on the blocks,
with a group of rotations and reflections of their cyclic order.

Let $\Lambda$ be the subgroup of $\Gamma$ consisting of rotations.
It is cyclic and has index at most two.
Every nonidentity rotation moves every block, and therefore
every vertex of $G/H$. Thus each $\Lambda$-orbit in $G/H$
has size $|\Lambda|$. Since $\Gamma$ acts transitively and
$[\Gamma:\Lambda]\le2$, there are at most two such orbits.
If there is exactly one orbit, the action of $\Lambda$ identifies $G/H$
with a connected Cayley graph of a cyclic group.
For sufficiently large $R_0$, we have
$|G/H|=n/m\ge2n/k\ge3$.
Using the result of Maru{\v{s}}i{\v{c}}~\cite{Marusic1983}
that every connected Cayley graph of a finite abelian group
of order at least three has a Hamilton cycle,
we obtain a Hamilton cycle in $G/H$.

Suppose that $\Lambda$ has two orbits in $G/H$.
Then $|G/H|=2|\Lambda|$, and transitivity forces $\Gamma$
to contain a reflection. Hence $\Gamma$ is a dihedral group
of order $2|\Lambda|$. Fix a vertex $v_0\in V(G/H)$.
The map $g\mapsto g(v_0)$ is a bijection from $\Gamma$
to $V(G/H)$, under which the action of $\Gamma$ becomes
left multiplication.
Thus $G/H$ is a connected Cayley graph of $\Gamma$
with a symmetric generating set $\mathcal A$.
We check the conditions of Lemma~\ref{lem:dihedral-cycle}
using the cyclic structure of the blocks.
Choose a generator $\rho$ of $\Lambda$ whose action on the
blocks is a rotation by the smallest positive number of positions in their cyclic order.
Let $B_0$ be the block containing the $H$-orbit $v_0$.
If $\mathcal A$ contains a rotation, it sends $B_0$
to a block at cyclic distance at most $t$ from $B_0$.
Writing this rotation as $\rho^a$ with $|a|$ minimal therefore
gives $0<|a|\le t$, as required in Lemma~\ref{lem:dihedral-cycle}~\textup{(i)}.
Otherwise, since $\mathcal A$ generates $\Gamma$, it contains
two distinct reflections $\sigma,\tau$.
Both $\sigma(B_0)$ and $\tau(B_0)$ are at cyclic distance at
most $t$ from $B_0$, and hence at distance at most $2t$ from
each other.
Thus $\tau\sigma=\rho^a$ for some $0<|a|\le2t$.
Taking $r=\rho$ and $f=\sigma$, the generators $f$ and
$r^af=\tau$ satisfy Lemma~\ref{lem:dihedral-cycle}~\textup{(ii)}.

Since $m,t\le k/2$ and $k\le5n^{1/5}$, we also have
\[
 |\Lambda|=\frac{n}{2m}\ge\frac nk>2k^2\ge8t^2
\]
for sufficiently large $R_0$.
Lemma~\ref{lem:dihedral-cycle} therefore gives
$c(G/H)\ge n/m-4t$.
The same inequality holds in the cyclic case.
Writing $q=c(G/H)$, we obtain
$qm\ge n-4mt\ge n-2k$.

We now apply Lemma~\ref{lem:normal-quantitative} with $\ell=k$.
Since each edge joins blocks at cyclic distance at most $t$,
each ball of radius $2k$ meets at most $4kt+1$ blocks. Thus
$\beta_{2k}\le s(4kt+1)=O(k^2)$ by $st\le k/2$.
Together with the inequalities $k\le5R^{1/4}<5n^{1/5}$, $m\le k/2$ and $qm\le n$, the lemma gives
\[
 c(G)\ge qm-8m\sqrt{qm\beta_{2k}}
 \ge n-O(k^2\sqrt n)
 \ge n-n^{19/20}.\qedhere
\]
\end{proof}

\subsection{Proof of Theorem~\ref{thm:main}}
It suffices to prove the theorem for sufficiently large $n$.
Let $R_0$ be as in Theorem~\ref{prop:early}, and let $G$ be a connected
vertex-transitive graph on $n$ vertices.
Write $d$ for its degree.
If $\beta_r\le r^{5/4}$ for some integer $R_0\le r<n^{4/5}$,
Theorem~\ref{prop:early} gives $c(G)\ge n-n^{19/20}$.
Since $d<n$, we have $n^{-1/20}\le d^{-1/100}$, giving
both required bounds for sufficiently large $n$.

Otherwise, fix $S\subseteq V(G)$, and order the components
$X_1,\ldots,X_t$ of $G-S$ by nonincreasing size.
We now verify~\eqref{eq:component-condition}.
For each integer $i\ge0$ with $2^i\le n/2$, take
$r=\max\{R_0,\lceil2^{4i/5}\rceil\}$.
For sufficiently large $n$, we have $R_0\le r<n^{4/5}$,
so $\beta_r>r^{5/4}\ge2^i$.
By Lemma~\ref{lem:ballcap},
\[
 \bigl|\{j\ge2:2^i\le|X_j|<2^{i+1}\}\bigr|
 \le\frac{(4r+3)|S|}{2^i}
 \le (4R_0+7)2^{-i/5}|S|.
\]
For $2^i\le|X_j|<2^{i+1}$, we have
$|X_j|^{1/8}\le2^{(i+1)/8}$, so
\[
 \sum_{j=2}^t|X_j|^{1/8}
 \le2^{1/8}(4R_0+7)|S|\sum_{i=0}^{\infty}2^{-3i/40}.
\]
The geometric series converges, so~\eqref{eq:component-condition}
holds with an absolute constant $K$.
Together with the first case, Theorem~\ref{thm:general}
therefore gives an absolute constant $c>0$ and an absolute
integer $d_0$ such that $G$ contains a cycle of length
at least $cn$ and, if $d\ge d_0$, a cycle of length
at least $(1-d^{-1/100})n$.\qed

\section{Concluding remarks}\label{sec:robust}
The condition \eqref{eq:component-condition} in
Theorem~\ref{thm:general} provides a unified approach to
constructing long cycles.
For vertex-transitive graphs, the dichotomy on the sizes of balls leads either to a verification of this condition
or to the structural argument in
Section~\ref{sec:vertex-transitive}.
For regular graphs satisfying~\eqref{eq:component-condition},
we can start with disjoint cycles covering almost all vertices, and the final cycle also covers almost all vertices as the degree tends to infinity.
It remains open whether every connected vertex-transitive graph on $n$ vertices admits a cycle of length $(1-o(1))n$.

We conclude with another application of Theorem~\ref{thm:general},
to regular robust sublinear expanders.
We recall the robust expansion condition introduced by Haslegrave, Kim
and Liu~\cite[Definition~3.1]{HKL}. For $\varepsilon_1,k>0$ and $x\ge k/2$,
let
\[
 \rho(x)=\frac{\varepsilon_1}{\log^2(15x/k)},
\]
where logarithms are natural. A $d$-regular graph $G$ is an
\emph{$(\varepsilon_1,k)$-robust expander} if, for every
$A\subseteq V(G)$ with $k/2\le |A|\le |G|/2$ and every
$F\subseteq E(G)$ with $|F|\le d\rho(|A|)|A|$, the external
neighbourhood of $A$ in $G-F$ has size at least $\rho(|A|)|A|$. Chen, Liu, Wei and Yang~\cite{CLWY} constructed regular sublinear expanders with small circumference. They asked whether there are fixed $\varepsilon_1,\varepsilon_2>0$ and a sequence of $n$-vertex $d$-regular $(\varepsilon_1,\varepsilon_2d)$-robust expanders, with $n\to\infty$, whose circumference is $o(n)$~\cite[Problem~5.4]{CLWY}. 

The following consequence of Theorem~\ref{thm:general} answers this question negatively.
For disjoint vertex sets $A,B$ in a graph $H$, let $e_H(A,B)$ denote the number
of edges in $H$ between $A$ and $B$.
\begin{theorem}\label{thm:robust}
For every $\varepsilon_1,\varepsilon_2>0$, there exist $c>0$ and an integer $n_0$ such that every $n$-vertex $d$-regular
$(\varepsilon_1,\varepsilon_2d)$-robust expander $G$ with $n\ge n_0$
and $d\ge1$ satisfies $c(G)\ge cn$.
\end{theorem}

\begin{proof}
Fix $\varepsilon_1,\varepsilon_2>0$, and let $G$ be an $n$-vertex
$d$-regular $(\varepsilon_1,\varepsilon_2d)$-robust expander, where
$d\ge1$ and $n$ is sufficiently large in terms of
$\varepsilon_1$ and $\varepsilon_2$.
Let $k=\varepsilon_2d$.
If $n\le2k$, then $d\ge2$ for sufficiently large $n$.
An endpoint of a longest path has all its $d$ neighbours on that
path. Its farthest neighbour along the path gives a cycle of
length at least $d+1$, so $c(G)\ge d+1\ge n/(2\varepsilon_2)$.

We then assume that $n>2k$. Choose $0<\eta\le1$, depending only on
$\varepsilon_1,\varepsilon_2$, so that
\begin{equation}\label{eq:rho-lower-bound}
 \rho(x)x\ge\eta\bigl(k+x^{1/8}\bigr)
 \qquad(x\ge k/2).
\end{equation}
Here $\eta$ can be chosen independently of $d$. Indeed, writing $x=ky$ with $y\ge1/2$, we have
\[
 \frac{\rho(x)x}{k+x^{1/8}}
 \ge \frac{\varepsilon_1y}
 {\log^2(15y)(1+\varepsilon_2^{-7/8}y^{1/8})},
\]
since $k=\varepsilon_2d\ge\varepsilon_2$.
The minimum of the right-hand side over $y\ge1/2$ is positive
and depends only on $\varepsilon_1$ and $\varepsilon_2$.
Therefore, any $0<\eta\le1$ smaller than this minimum
satisfies~\eqref{eq:rho-lower-bound}.

The robust expansion property, together
with~\eqref{eq:rho-lower-bound}, gives
\begin{equation}\label{eq:robust-edge-cut}
 e_G(A,V(G)\setminus A)>d\rho(|A|)|A|
 \ge\eta d\bigl(k+|A|^{1/8}\bigr)
 \qquad(k/2\le |A|\le n/2).
\end{equation}
We next find an induced subgraph $H$ of $G$ with $|H|>3n/4$ such that
$e_H(A,V(H)\setminus A)\ge\eta d|A|^{1/8}$ whenever
$1\le|A|\le|H|/2$.
Choose $R\subseteq V(G)$ minimizing
$e_G(R,V(G)\setminus R)-\eta d|R|$ subject to $|R|\le k$.
Comparison with $R=\emptyset$ gives
$e_G(R,V(G)\setminus R)\le\eta d|R|$.
If $k/2\le|R|\le k$, then~\eqref{eq:robust-edge-cut} would imply $e_G(R,V(G)\setminus R)>\eta d k\ge\eta d|R|$, a contradiction.
Thus $|R|<k/2$, and $H:=G-R$ has more than $3n/4$ vertices because $k<n/2$.

For a vertex set $A\subseteq V(H)$ with $1\le|A|\le k/2$, since $|R|<k/2$, we may compare $R$ with $R\cup A$ in the
minimization. Adding $A$ to $R$ removes $e_G(A,R)$ edges from
the boundary and adds $e_H(A,V(H)\setminus A)$ edges to it.
Minimality of $R$ therefore gives
\[
 e_H(A,V(H)\setminus A)
 \ge e_G(A,R)+\eta d|A|
 \ge\eta d|A|^{1/8}.
\]
For $A\subseteq V(H)$ with $k/2<|A|\le|H|/2$,
Since $H=G-R$ and
$e_G(A,R)\le e_G(R,V(G)\setminus R)\le\eta d|R|$,
we obtain from~\eqref{eq:robust-edge-cut} that
\[
\begin{aligned}
 e_H(A,V(H)\setminus A)=e_G(A,V(G)\setminus A)-e_G(A,R)>\eta d\bigl(k+|A|^{1/8}\bigr)-\eta d|R|\ge\eta d|A|^{1/8},
\end{aligned}
\]
where the last inequality uses $|R|<k/2$.
Therefore, $e_H(A,V(H)\setminus A)
 \ge\eta d|A|^{1/8}$ holds for every $A\subseteq V(H)$ of size at most $|H|/2$.

We now verify the condition~\eqref{eq:component-condition}.
For $S\subseteq V(H)$, order the components $X_1,X_2,\ldots$
of $H-S$ by nonincreasing size. Since $|X_j|\le|H|/2$ for each $j\ge2$, we have
\[
 \sum_{j\ge2}|X_j|^{1/8}
 \le\frac{1}{\eta d}\sum_{j\ge2}e_H(X_j,S)
 \le \eta^{-1}|S|,
\]
where the second inequality holds because every vertex of $S$ has degree at most $d$ in $H$.
Applying Theorem~\ref{thm:general} with $K=\eta^{-1}$, we deduce
$c(G)\ge c(H)\ge c_0|H|>3c_0n/4$ for sufficiently large $n$,
where $c_0>0$ depends only on $\varepsilon_1,\varepsilon_2$.
Together with the case $n\le2k$, this proves the theorem with $c=\min\{3c_0/4,1/(2\varepsilon_2)\}$.
\end{proof}

\section*{AI Declaration}
ChatGPT (GPT-6) was used to assist with the exploration of
proof strategies, the formulation and testing of intermediate
claims, and the preparation of this manuscript.
The current project is a continuation of our earlier study of longest cycles
in vertex-transitive graphs~\cite{MaZhao}, where we used a structural theorem
of DeVos and Mohar~\cite{DeVosMohar2006} to investigate the intersection of
longest cycles in such graphs.
Earlier attempts with models in the GPT-5 series did not yield improvements
toward Lov\'asz's conjecture.
Since early September 2026, we resumed the exploration with GPT-6.
This exploration had led to a lower bound $c(G)=\Omega(n/(\log n)^4)$ for every connected vertex-transitive graph $G$ on $n$ vertices.
The distinction according to the sizes of balls at suitable
values of the radius emerged during these exchanges and was new to us.
When the size of a ball was bounded above by a suitable
polynomial in its radius,
we used a structure theorem of Tessera and
Tointon~\cite{TesseraTointon}. Our approach in this case was to
obtain a quotient graph with restricted algebraic structure,
find a long cycle in that quotient graph, and lift it to the original graph.
In the complementary case, inspired by the cycle-intersection
pattern in~\cite[Lemma~2.8]{BHMSY}, we used the Lov\'asz local lemma
to choose paths whose union contained a long cycle.
Iteratively refining this approach led to the bound
$n/(\log n)^{1+o(1)}$.
We then noticed the paper of Segal and
Verstra\"ete~\cite{SegalVerstraete2026} on tough graphs.
After reading it, we found it natural to explore its connection with our
approach in further discussions with GPT-6.
In the case where the sizes of balls have polynomial lower bounds,
an adaptation of their argument replaced our earlier use of the
Lov\'asz local lemma and yielded a linear bound.
We also followed their application of Mader's theorem and other results on
disjoint paths. These tools eventually led to Theorem~\ref{thm:general},
giving a general criterion for finding long cycles.
For the case governed by a polynomial upper bound,
we retained the strategy of finding a long cycle in a
suitable quotient graph and lifting it to the original graph.
We also managed to replace the Tessera--Tointon input with the graph-theoretic result stated as Theorem~\ref{thm:DM},
from the structural work of DeVos and
Mohar~\cite{DeVosMohar}.
For regular graphs, we found that starting from a family of disjoint cycles
covering almost all vertices then yields a cycle of length $(1-o_d(1))n$
as the degree $d$ tends to infinity.
Combining all these ideas led to the present results and proofs.

The authors directed this process, independently checked
the proofs, and take full responsibility for the
mathematical content and presentation.

\begingroup
\renewcommand{\bibfont}{\small}
\setlength{\bibsep}{2pt}

\endgroup
\appendix

\section{Interpolating condition~\eqref{eq:component-condition} between toughness and expansion}
\label{app:toughness-expansion}

We justify the comparisons with toughness and vertex
expansion stated after Theorem~\ref{thm:general}.
Let $G$ be a graph on $n\ge3$ vertices, let $h$ be a
positive function on the positive integers, and let $K\ge0$.
We consider condition~\eqref{eq:weighted-components},
which requires
\[
 \sum_{j=2}^t h(|X_j|)\le K|S|
\]
for every $S\subseteq V(G)$, where $X_1,\ldots,X_t$ are
the components of $G-S$ in nonincreasing order of size.
The sum is zero when $G-S$ has at most one component.
Taking $S=\emptyset$ shows that $G$ is connected. Moreover, when $K=0$, the condition holds precisely when $G$ is complete.

Recall that a graph is \emph{$\tau$-tough}, where $\tau>0$,
if $|S|\ge\tau t$ whenever $G-S$ has $t\ge2$
components (see~\cite{Chvatal}).
Suppose that $K>0$ and $h\equiv1$.
For every such set $S$, condition~\eqref{eq:weighted-components}
gives $t-1\le K|S|$.
Since $t\ge2$ and $|S|\ge1$, we have
\[
 t\le2(t-1)\le2K|S|,
 \qquad
 t\le K|S|+1\le(K+1)|S|.
\]
Thus $G$ is $\max\{1/(2K),1/(K+1)\}$-tough.
Conversely, if $G$ is $\tau$-tough, then
$t-1\le t\le |S|/\tau$ whenever $t\ge2$.
Hence condition~\eqref{eq:weighted-components} holds with
$h\equiv1$ and $K=1/\tau$.
Thus, for $K>0$, condition~\eqref{eq:weighted-components}
with $h\equiv1$ is equivalent to $1/K$-toughness up to
a factor of two in the parameter.

We say that $G$ has vertex expansion at least
$\varepsilon>0$ if every set $A\subseteq V(G)$ with
$1\le|A|\le n/2$ has at least $\varepsilon|A|$ neighbors
outside $A$.
We show that condition~\eqref{eq:weighted-components}
with $h(x)=x$ implies vertex expansion at least
$1/(K+1)$.
For such a set $A$, let $S$ be the set of its neighbors
outside $A$, and let $B=V(G)\setminus(A\cup S)$.
Every component of $G-S$ lies entirely in $A$ or in $B$.
Since $|B|=n-|A|-|S|\ge|A|-|S|$, we have
\[
 |A|-|S|
 \le\min\{|A|,|B|\}
 \le\sum_{j=2}^t|X_j|
 \le K|S|.
\]
Consequently, $|S|\ge|A|/(K+1)$, so $G$ has vertex expansion at least $1/(K+1)$.

Conversely, suppose that $G$ has vertex expansion at least
$\varepsilon$.
We prove condition~\eqref{eq:weighted-components} with $h(x)=x$
and $K=2/\varepsilon$.
Fix $S$ such that $G-S$ has at least two components.
Partition these components into two nonempty families
whose unions $A$ and $B$ minimize the difference between
their sizes, with $|A|\le|B|$.
Then $|B|-|A|\le|X_1|$: otherwise, moving any component
from $B$ to $A$ would reduce this difference.
Therefore,
\[
 \frac12\sum_{j=2}^t|X_j|
 =\frac{n-|S|-|X_1|}{2}
 \le |A|.
\]
Every neighbor of $A$ outside $A$ belongs to $S$,
so $\varepsilon|A|\le|S|$.
Hence $\sum_{j=2}^t|X_j|\le2|S|/\varepsilon$, as required.
In particular, for $0<\varepsilon\le1$, these two
implications compare the expansion parameter with
$1/(K+1)$ up to absolute constant factors.
\end{document}